\pdfoutput=1
\documentclass[reqno]{amsart}

\usepackage{amsfonts}
\usepackage{amsmath}
\usepackage{amssymb}
\usepackage{mathtools}
\usepackage{bm}
\usepackage{enumitem}
\usepackage{bookmark}
\usepackage[square,numbers,sort&compress]{natbib}
\usepackage{hyperref}
\hypersetup{colorlinks=true,linkcolor=blue,citecolor=red}
\allowdisplaybreaks

\def\Xint#1{\mathchoice
{\XXint\displaystyle\textstyle{#1}}
{\XXint\textstyle\scriptstyle{#1}}
{\XXint\scriptstyle\scriptscriptstyle{#1}}
{\XXint\scriptscriptstyle\scriptscriptstyle{#1}}
\!\int}
\def\XXint#1#2#3{{\setbox0=\hbox{$#1{#2#3}{\int}$ }
\vcenter{\hbox{$#2#3$ }}\kern-.6\wd0}}

\newtheorem{theorem}{Theorem}[section]
\newtheorem{lemma}[theorem]{Lemma}

\newtheorem{remark}[theorem]{Remark}
\newtheorem{assumption}[theorem]{Assumption}
\theoremstyle{definition}
\newtheorem{definition}[theorem]{Definition}

\numberwithin{equation}{section}

\newcommand{ \mc }{ \mathcal }
\newcommand{ \mr }{ \mathbb{R} }

\newcommand{ \bu }{ \mathbf{u} }
\newcommand{ \bw }{ \mathbf{w} }
\newcommand{ \bv }{ \mathbf{v} }
\newcommand{ \bz }{ \mathbf{z} }
\newcommand{ \bV }{ \mathbf{V} }
\newcommand{ \bA }{ \mathbf{A} }

\newcommand{\ds}{\displaystyle}
\newcommand{\plog}{\mc{P}_{\log^{\gamma}}}

\begin{document}
\title[Borderline double phase functionals with powers of the logarithm]{Borderline double phase functionals with powers of the logarithm: regularity in the vectorial case}

\author{Youlin Choi}\address{Department of Mathematics, Kyungpook National University, Daegu, 41566, Republic of Korea} \email{kdbfls1268@knu.ac.kr} \author{Jehan Oh}\address{Department of Mathematics, Kyungpook National University, Daegu, 41566, Republic of Korea} \email{jehan.oh@knu.ac.kr}

\date{\today.}

\begin{abstract}
    We study vector-valued local minimizers of the borderline double phase functional whose integrand is $|D\bw|^p+a(x)|D\bw|^p\log^\gamma(e+|D\bw|)$, where $p>1$, $\gamma>0$ and the coefficient $a(\cdot)$ is nonnegative and bounded. If the modulus of continuity $\omega(\cdot)$ of $a(\cdot)$ satisfies $\omega(r)\log^\gamma(1/r)\to0$ as $r\to0$, then minimizers are locally H\"older continuous for every exponent less than one. If $a(\cdot)$ is H\"older continuous, then the gradients of minimizers are locally H\"older continuous. In the vector-valued setting these results were previously available only for $\gamma=1$. The proof combines a comparison with frozen functionals, whose regularity theory applies uniformly in the frozen coefficient, with the higher integrability of minimizers.
\end{abstract}
\maketitle

\section{\bf Introduction}\label{section 1}

Research on double phase functionals can be traced back to the work of Zhikov in the 1980s \cite{MR864171}. Since then, the regularity theory for variational problems with nonstandard growth has developed substantially. In the early study of such non-autonomous functionals, Zhikov \cite{MR1486765} considered three model functionals in connection with the Lavrentiev phenomenon, namely
\begin{equation} \label{eq: 1.1}
    \begin{aligned} 
        \mc{M}(w,\Omega) &:=\int_\Omega c(x)|Dw|^2\,dx, \qquad 0<c^{-1}\in L^t(\Omega), \quad t>1;\\ 
        \mc{V}(w,\Omega) &:=\int_\Omega |Dw|^{p(x)}\,dx, \qquad 1<p(x)<\infty;\\ 
        \mc{P}_{p,q}(w,\Omega) &:=\int_{\Omega} \left(|Dw|^p+a(x)|Dw|^q\right)\,dx, \qquad 0\leq a(x)\leq L,\quad 1<p<q. 
    \end{aligned} 
\end{equation} 
These functionals provide fundamental examples of non-autonomous variational problems with nonstandard growth and have played an important role in the development of the theory of the Lavrentiev phenomenon and of non-uniform ellipticity \cite{MR1350506, MR1486765}.

To place these developments in context, we first recall functionals with classical $p$-growth \cite{MR727034, MR2566733}. The prototypical example is
\begin{equation*} 
    \mathcal{F}_p(w,\Omega) := \int_{\Omega}|Dw|^p\,dx, \qquad p>1, 
\end{equation*} 
and the corresponding Euler--Lagrange equation is of $p$-Laplace type. The regularity theory for classical $p$-growth functionals has served as a fundamental starting point in the study of the regularity of nonlinear variational problems and partial differential equations. Subsequently, the theory was extended to more general non-autonomous functionals, in which the integrand depends explicitly on the spatial variable \cite{MR1240398, MR1895714},
\begin{equation*} 
    \mathcal{F}(w,\Omega) := \int_{\Omega}F(x,Dw)\,dx, \qquad \Omega\subset\mr^n.
\end{equation*} 
In such non-autonomous functionals, the growth and the ellipticity of the integrand may vary with the spatial variable $x$, and this gives rise to regularity problems that are distinct from those of the classical $p$-growth setting.

A representative example of such non-autonomous functionals is the variable exponent functional \cite{MR3026953}
\begin{equation*} 
    \mathcal{V}(\bw,\Omega) := \int_{\Omega}|D\bw|^{p(x)}\,dx,
\end{equation*}
where the growth rate is determined by an exponent $p(x)$ that varies with the spatial variable, rather than by a fixed exponent $p$. In particular, the relationship between the continuity of $p(\cdot)$ and the regularity of minimizers has been studied extensively, and it is known that the gradient of a minimizer is H\"older continuous whenever $p(\cdot)$ is H\"older continuous \cite{MR1675954}.

On the other hand, a double phase functional, in which two different growth rates appear in a single functional, is given by
\begin{equation*} 
    \mathcal{P}_{p,q}(\bw,\Omega) :=\int_{\Omega} \left( |D\bw|^p + a(x)|D\bw|^q \right)\,dx,
\end{equation*}
where $1<p\leq q$ and $a(\cdot)$ is a nonnegative coefficient function \cite{MR3294408, MR3360738, MR3348922}. The regularity theory for functionals with $(p,q)$-growth was developed in a series of works, beginning with the work of Marcellini and subsequently extended to sharper results \cite{MR969900, MR1094446, MR1240398, MR2076158}. The two terms have different growth rates, and the coefficient $a(\cdot)$ determines the transition between the two phases as the spatial variable varies. The regularity theory for this double phase structure has since been developed extensively, with particular emphasis on the continuity of $a(\cdot)$ and on the relation between the two growth exponents $p$ and $q$ \cite{MR3294408, MR3360738, MR3775180, MR4149062}.

Along with such double phase growth, growth conditions involving a logarithmic perturbation of the $p$-growth have also been studied \cite{MR3570955}. A representative example is the functional
\begin{equation*} 
    \mathcal{L}_p(\bw,\Omega) :=\int_{\Omega} |D\bw|^p\log(e+|D\bw|)\,dx.
\end{equation*} 
This logarithmic perturbation is a borderline case of $p$-growth and arises naturally in the study of nonstandard growth problems \cite{MR3570955}.

In particular, the following functional, which combines logarithmic growth with the double phase structure, has been considered in \cite{MR3570955}:
\begin{equation*} 
    \mathcal{P}_{\log}(\bw,\Omega) :=\int_{\Omega} \left[ |D\bw|^p + a(x)|D\bw|^p\log(e+|D\bw|) \right]\,dx. 
\end{equation*}
Compared with the general double phase functional with integrand
\begin{equation*}
    |D\bw|^p+a(x)|D\bw|^q,
\end{equation*}
this functional may be viewed as a borderline case of double phase growth, and the regularity of its minimizers has been studied in \cite{MR3570955}. In this case, what matters is not only the H\"older continuity of $a(\cdot)$ but also the interaction between the modulus of continuity $\omega(\cdot)$ of $a(\cdot)$ and the logarithmic growth. Specifically, under the condition
\begin{equation*}
    |a(x)-a(y)| \leq \omega(|x-y|),
\end{equation*}
the quantity
\begin{equation*}
    \limsup_{r\to0} \omega(r)\log\left(\frac{1}{r}\right)
\end{equation*}
governs the regularity of minimizers and of their gradients \cite{MR3570955}.

More generally, the logarithmic factor may be replaced by a power of it, which leads to the growth
\begin{equation*} 
    |Dw|^p\log^\gamma(e+|Dw|).
\end{equation*} 
From this point of view, in this paper we consider the more general vector-valued logarithmic double phase functional
\begin{equation} \label{eq: 1.2}
    \mathcal{P}_{\log^\gamma}(\bw,\Omega) := \int_{\Omega}
    \left[ |D\bw|^p +  a(x)|D\bw|^p\log^\gamma(e+|D\bw|) \right]\,dx.
\end{equation}
Here $\gamma>0$, and the case $\gamma=1$ corresponds to the logarithmic double phase functional mentioned above. Thus, \eqref{eq: 1.2} describes a one-parameter family of borderline double phase growth conditions, all of which extend the classical $p$-growth.

In the scalar-valued setting, the regularity of minimizers of this functional has already been studied. Two regularity results follow from the general theory of \cite{MR4397041}, and a further one from \cite{MR4288665}. Combining them, we obtain the following three statements for scalar-valued minimizers.

Let $\Omega \subset \mr^n$ be a bounded open set and consider the functional $\mc{P}_{\log^{\gamma}}$ defined by  
\begin{equation}  
    \label{eq: 1.3}  
    W^{1,1}(\Omega) \ni w \mapsto \plog(w,\Omega):= \int_{\Omega}[|Dw|^p+a(x)|Dw|^p\log^{\gamma}(e+|Dw|)]\,dx.  
\end{equation} 
Here the function $a(\cdot)$ is nonnegative and bounded, $p>1$ and $\gamma>0$. 
\begin{theorem} \label{thm: 1.1}
    Let $u \in W^{1,1}(\Omega)$ be a local minimizer of the functional $\plog$ defined in \eqref{eq: 1.3}. Assume that $\omega(\cdot)$ is a modulus of continuity of $a(\cdot)$, that is, for every $x,y \in \Omega$,  
    \begin{equation} \label{eq: 1.4} 
        |a(x)-a(y)|\leq \omega(|x-y|).
    \end{equation}
    Define  
    \begin{equation}  \label{eq: 1.5}
        \ell:=\limsup_{r\to0}\omega(r)\log^{\gamma}\left( \frac{1}{r} \right).  
    \end{equation}  
    Then the following statements hold.  
    \begin{enumerate}[label=\upshape{(\roman*)}]  
        \item \label{thm: 1.1(1)} If $\ell <\infty$, then $u \in C_{\mathrm{loc}}^{0,\beta}(\Omega)$ for some $\beta \in (0,1)$.  
        \item \label{thm: 1.1(2)} If $\ell =0$, then $u \in C_{\mathrm{loc}}^{0,\beta}(\Omega)$ for every $\beta \in (0,1)$.  
        \item \label{thm: 1.1(3)} If $\omega(r)\lesssim r^{\sigma}$ with $\sigma \in (0,1)$, then $Du$ is locally H\"older continuous in $\Omega$.  
    \end{enumerate}  
\end{theorem} 

In the vector-valued setting, however, the three statements above do not all hold in the same form. Indeed, everywhere regularity cannot be expected in general for vector-valued variational problems, as classical counterexamples show, see for example \cite{MR3099262}. Thus, in the vector-valued setting, the statements \ref{thm: 1.1(2)} and \ref{thm: 1.1(3)} of Theorem~\ref{thm: 1.1} are the natural objects of study. In the case $\gamma=1$, both have been established in \cite{MR3570955}. We also refer to \cite{baroni2022gradient}, where the $C^1$ regularity of vector-valued minimizers is obtained for $\gamma=1$ under a log-Dini continuity assumption on $a(\cdot)$, and to \cite{ragusa2024regularity}, where the variable exponent version of the borderline functional, with $|D\bw|^{p(x)}$ in place of $|D\bw|^p$, is treated for vector-valued minimizers. To the best of our knowledge, no regularity result for vector-valued minimizers of \eqref{eq: 1.2} with $\gamma\neq1$ is available in the literature.

In this respect, we point out that the vectorial case is explicitly left open in the theory of H\"ast\"o and Ok \cite{MR4397041} for functionals with generalized Orlicz growth, which contains the scalar case of \eqref{eq: 1.2}. In \cite[Remark~1.7]{MR4397041} the main difficulty is identified as follows. In order that local minimizers of the autonomous functional with frozen Orlicz function $\tilde{\varphi}$ have $C^{1,\alpha}$ regularity, $\tilde{\varphi}$ should satisfy not only $t\tilde{\varphi}''(t)\approx\tilde{\varphi}'(t)$ but also a H\"older type condition on $\tilde{\varphi}''$, as required in \cite[Assumption~2.2]{diening2009everywhere}, and it is unclear whether such a condition follows from the continuity assumptions imposed on the integrand. In the present setting, the frozen functionals are explicit. Freezing the coefficient $a(\cdot)$ at its infimum $a_0\geq0$ over a ball leads to the Orlicz function $t\mapsto t^p+a_0t^p\log^\gamma(e+t)$, and both conditions hold with constants depending only on $p$ and $\gamma$, in particular independent of $a_0$, see Remark~\ref{rmk: 2.9}. This is what makes the comparison argument of \cite{MR3570955} available for every $\gamma>0$. Motivated by this, we investigate whether the two regularity properties above continue to hold for every $\gamma>0$ for vector-valued local minimizers of \eqref{eq: 1.2}. Our analysis is governed by the interaction between $\gamma$ and the modulus of continuity of $a(\cdot)$.

We are now in a position to state the main result. Consider the functional $\mc{P}_{\log^\gamma}$ defined in \eqref{eq: 1.2}, where $a(\cdot)$ is nonnegative and bounded, $p>1$ and $\gamma>0$. Following \cite{MR3570955}, a function $\bu\in W^{1,1}(\Omega;\mr^N)$ is called a local minimizer of $\mc{P}_{\log^\gamma}$ if $|D\bu|^p+a(x)|D\bu|^p\log^\gamma(e+|D\bu|)\in L^1_{\mathrm{loc}}(\Omega)$ and
\begin{equation*}
    \mc{P}_{\log^\gamma}(\bu,\operatorname{supp}(\bu-\bw)) \leq \mc{P}_{\log^\gamma}(\bw,\operatorname{supp}(\bu-\bw))
\end{equation*}
holds for every $\bw\in W^{1,1}(\Omega;\mr^N)$ with $\operatorname{supp}(\bu-\bw)\Subset\Omega$. Local minimizers of the other functionals considered in Sections~\ref{section 2} and \ref{section 3} are defined in the same way. 
 
\begin{theorem} \label{thm: 1.2} 
    Let $\bu\in W^{1,1}(\Omega;\mr^N)$ be a vector-valued local minimizer of the functional $\mc{P}_{\log^\gamma}$ defined in \eqref{eq: 1.2}. Assume that $\omega(\cdot)$ is a modulus of continuity of $a(\cdot)$, that is, for every $x,y\in\Omega$, 
    \begin{equation} \label{eq: 1.6} 
        |a(x)-a(y)| \leq \omega(|x-y|). 
    \end{equation} 
    Define 
    \begin{equation} \label{eq: 1.7} 
        \ell := \limsup_{r\to0} \omega(r)\log^\gamma\left(\frac{1}{r}\right). 
    \end{equation} 
    Then the following statements hold. 
    \begin{enumerate}[label=\upshape{(\roman*)}]  
        \item \label{thm: 1.2(1)} If $\ell =0$, then $\bu \in C_{\mathrm{loc}}^{0,\beta}(\Omega;\mr^N)$ for every $\beta \in (0,1)$.  
        \item \label{thm: 1.2(2)} If $\omega(r)\lesssim r^{\sigma}$ with $\sigma \in (0,1)$, then $D\bu$ is locally H\"older continuous in $\Omega$.
    \end{enumerate}  
\end{theorem}
The proof of Theorem~\ref{thm: 1.2} is based on a comparison argument, following the approach developed in \cite{MR3570955} for the case $\gamma=1$. More precisely, on each small ball $B_R\equiv B_R(x_0)\Subset\Omega$, we compare the minimizer $\bu$ with the minimizer of the frozen functional obtained by replacing the coefficient $a(\cdot)$ with its infimum $a_i(R):=\inf_{B_R}a(\cdot)$, and we control the resulting difference in terms of the modulus of continuity $\omega(\cdot)$ of $a(\cdot)$. For general $\gamma>0$, the interaction between the logarithmic factor $\log^\gamma(e+|D\bu|)$ and $\omega(\cdot)$ has to be handled carefully, and the quantity $\omega(R)\log^\gamma(1/R)$ plays a crucial role in this process. In particular, by combining the higher integrability of minimizers with estimates in the Zygmund-type Orlicz space $L^p\log^\gamma L$, we show that the error arising in the comparison step is controlled exactly by $\omega(R)\log^\gamma(1/R)$.

The paper is organized as follows. In Section~\ref{section 2}, we fix the notation and collect basic properties of the Zygmund-type Orlicz space $L^p\log^\gamma L$ and of $N$-functions, together with regularity estimates for minimizers of the frozen functionals, in which the coefficient $a(\cdot)$ is constant. Section~\ref{section 3} is devoted to a Caccioppoli type inequality, from which we derive the higher integrability of minimizers by means of Gehring's theory. In Section~\ref{section 4}, we prove Theorem~\ref{thm: 1.2}. In Subsection~\ref{subsection 4.1} we first establish a decay estimate by means of an energy comparison and a comparison lemma. In Subsection~\ref{subsection 4.2}, under the assumption $\ell=0$, we prove that minimizers are H\"older continuous for every exponent $\beta\in(0,1)$ by using properties of Campanato spaces. Finally, in Subsection~\ref{subsection 4.3}, we prove the H\"older continuity of the gradient of minimizers when $a(\cdot)$ is H\"older continuous.

\section{Preliminaries}\label{section 2}
\subsection{\bf Notation}\label{subsection 2.1}
We denote by $c\geq1$ a generic constant, which may vary from line to line. We also use variants such as $c_1$, $c_*$ and $\bar{c}$. To emphasize the dependence on parameters, we use parentheses. For example, $c\equiv c(n,p,q)$ means that $c$ depends only on $n$, $p$ and $q$. For a center $x_0\in\mr^n$ and a radius $r>0$, we denote the open ball by
\begin{equation} \label{eq: 2.1}
    B_r(x_0):=\{x \in \mr^n \mid |x-x_0|<r\}.
\end{equation}
We simply write $B_r \equiv B_r(x_0)$ when the center is clear from the context or not important. If $\mc{B}\subset \mr^n$ is a measurable set with positive and finite measure and $g: \mc{B} \to \mr^k$, $k \geq 1$, is an integrable map, we denote its integral average by
\begin{equation}    \label{eq: 2.2}
    (g)_{\mc{B}} \equiv \Xint-_{\mc{B}} g(x) \,dx :=\frac{1}{|\mc{B}|} \int_{\mc{B}} g(x)\,dx.
\end{equation}
Throughout this paper, $\log x$ denotes the natural logarithm of $x > 0$. Furthermore, for any $\gamma > 0$, we define $\log^\gamma (e+x) := [\log(e+x)]^\gamma$. For any $\bz, \bw \in \mr^{N\times n}$, we denote the standard inner product by
\begin{equation} \label{eq: 2.3}
    \langle \bz, \bw \rangle := \text{tr}(\bz^T \bw) = \sum_{i=1}^N \sum_{j=1}^n z_{ij} w_{ij},
\end{equation}
and its associated Euclidean norm by $|\bz| := \sqrt{\langle \bz, \bz \rangle}$.

\subsection{\bf Orlicz spaces}\label{subsection 2.2}
Let $\Omega \subset \mr^n$ be a bounded open set and let $\varphi:[0,\infty)\to[0,\infty)$ be a convex and strictly increasing function satisfying
$$
\lim_{t \to 0} \frac{\varphi(t)}{t}=0 \qquad \text{and} \qquad \lim_{t \to \infty}\frac{\varphi(t)}{t}=\infty.
$$
We denote by $L^{\varphi}(\Omega)$ the corresponding Orlicz space endowed with the Luxemburg norm
\begin{equation}    \label{eq: 2.4}
    \|f\|_{\varphi}:=\inf \left\{ \lambda>0 \;:\; \Xint-_{\Omega} \varphi\left( \frac{|f|}{\lambda} \right)\,dx \leq 1 \right\}.
\end{equation}
We focus on the particular Orlicz space generated by $\varphi(t)=t^p\log^{\gamma}(e+t)$, where $p\geq1$ and $\gamma>0$. The corresponding Orlicz space is denoted by
$$
L^{p,\gamma}:=L^p\log^{\gamma}L,
$$
and is called a Zygmund-type Orlicz space. For $f\in L^p(B_R)$ with $f\not\equiv0$, we use the quantity
\begin{equation*}
    \Xint-_{B_R}|f|^p\log^{\gamma}\left( e+\frac{|f|^p}{(|f|^p)_{B_R}} \right)\,dx,
\end{equation*}
which, by \cite{iwaniec1999operator}, is comparable to $\|f\|_{L^{p,\gamma}}^p$ when the Luxemburg norm \eqref{eq: 2.4} is taken on $B_R$, with constants depending only on $p$ and $\gamma$. The only property of this quantity needed in what follows is the following elementary estimate. For every $q>1$ and every ball $B_\varrho$ concentric with $B_R$ such that $R/2\leq\varrho\leq R$, we have
\begin{equation}    \label{eq: 2.5}
    \Xint-_{B_\varrho}|f|^p\log^{\gamma}\left( e+\frac{|f|^p}{(|f|^p)_{B_R}} \right)\,dx \leq c(n,\gamma,q)\left( \Xint-_{B_\varrho}|f|^{pq}\,dx \right)^{\frac{1}{q}}.
\end{equation}
Indeed, set $g:=|f|^p/(|f|^p)_{B_R}$, so that $(g)_{B_\varrho}\leq 2^n(g)_{B_R}=2^n$. Then H\"older's inequality and the elementary inequality $\log^{\gamma q'}(e+g)\leq c(\gamma,q)(1+g)$, where $q'=q/(q-1)$, yield
\begin{equation*}
    \begin{aligned}
        \Xint-_{B_\varrho}g\log^{\gamma}(e+g)\,dx
        & \leq \left( \Xint-_{B_\varrho}g^q\,dx \right)^{\frac{1}{q}}\left( c(\gamma,q)\Xint-_{B_\varrho}(1+g)\,dx \right)^{\frac{1}{q'}} \\
        & \leq c(n,\gamma,q)\left( \Xint-_{B_\varrho}g^q\,dx \right)^{\frac{1}{q}},
    \end{aligned}
\end{equation*}
and \eqref{eq: 2.5} follows upon multiplying both sides by $(|f|^p)_{B_R}$. For further details on Zygmund-type Orlicz spaces, we refer the reader to \cite{cho2024zygmund}.

\subsection{\bf\texorpdfstring{$N$}{N}-functions}\label{subsection 2.3}
\begin{definition} \label{def: 2.1}
    A convex function $\varphi : [0,\infty) \to [0,\infty)$ is called an $N$-function if
    \begin{equation} \label{eq: 2.6}
        \begin{gathered}
            \varphi \in C^1([0,\infty))\cap C^2((0,\infty)), \; \varphi(0)=\varphi'(0)=0,\\
            \varphi(t)=0 \iff t=0,\\
            \varphi'(t)\text{ is monotone increasing and } \lim_{t\to\infty}\varphi'(t)=\infty.
        \end{gathered}
    \end{equation}
\end{definition}

\begin{definition}    \label{def: 2.2}
    Let $\bV_\varphi : \mr^{N\times n} \to \mr^{N\times n}$ be defined by 
    \begin{equation}    \label{eq: 2.7}
        \bV_\varphi (\bz):=\left( \frac{\varphi'(|\bz|)}{|\bz|} \right)^{1/2}\bz,
    \end{equation}
    for all $\bz \in \mr^{N\times n}$.
\end{definition}

By Definition~\ref{def: 2.1}, $\bV_\varphi$ is a bijection of $\mr^{N\times n}$.

\begin{assumption}  \label{asum: 2.3}
    Let $\varphi$ be an $N$-function. Assume that there exists $c_\varphi\geq1$ such that
    \begin{equation}    \label{eq: 2.8}
        \frac{1}{c_\varphi} \leq \frac{\varphi''(t)t}{\varphi'(t)} \leq c_\varphi, \quad \forall t>0.
    \end{equation}
\end{assumption}
\begin{remark}    \label{rmk: 2.4}
    From \eqref{eq: 2.8} we obtain, by integration, that for all $t\geq 0$
    \begin{equation}    \label{eq: 2.9}
        \frac{\varphi(t)}{c(c_\varphi)} \leq \varphi'(t)t \leq c(c_\varphi)\varphi(t).
    \end{equation}
\end{remark}

The following two lemmas are standard consequences of Assumption~\ref{asum: 2.3}, see \cite{diening2008fractional, diening2009everywhere}. The first one expresses the monotonicity of the map $\bz\mapsto\varphi'(|\bz|)\bz/|\bz|$ in terms of $\bV_\varphi$.

\begin{lemma}    \label{lem: 2.5}
    Let $\bV_\varphi : \mr^{N\times n} \to \mr^{N\times n}$ be defined by \eqref{eq: 2.7} and let $\varphi$ satisfy Assumption \ref{asum: 2.3}. Then there exists a constant $c\geq 1$, depending only on $n$ and $c_\varphi$, such that for all $\bz_1, \bz_2 \in \mr^{N\times n}$,
    \begin{equation}    \label{eq: 2.10}
        \left\langle \frac{\varphi'(|\bz_1|)}{|\bz_1|}\bz_1-\frac{\varphi'(|\bz_2|)}{|\bz_2|}\bz_2, \bz_1-\bz_2 \right\rangle \approx |\bV_\varphi (\bz_1) -\bV_\varphi (\bz_2)|^2.
    \end{equation}
\end{lemma}

\begin{lemma}    \label{lem: 2.6}
    Let $\bV_\varphi : \mr^{N\times n} \to \mr^{N\times n}$ be defined by \eqref{eq: 2.7} and let $\varphi$ satisfy Assumption \ref{asum: 2.3}. Then there exists a constant $c\geq 1$, depending only on $n$ and $c_\varphi$, such that for all $\bz_1, \bz_2 \in \mr^{N\times n}$,
    \begin{equation}    \label{eq: 2.11}
        \varphi''(|\bz_1|+|\bz_2|)|\bz_1-\bz_2|^2 \leq c|\bV_\varphi(\bz_1)-\bV_\varphi(\bz_2)|^2. 
    \end{equation}
    Moreover, as a direct consequence of \eqref{eq: 2.9}, we have
    \begin{equation}    \label{eq: 2.12}
        \frac{1}{c}|\bV_\varphi (\bz)|^2 \leq \varphi(|\bz|) \leq c|\bV_\varphi(\bz)|^2,
    \end{equation}
    for every $\bz \in \mr^{N\times n}$, where $c\equiv c(c_\varphi)$.
\end{lemma}

For the choice $\varphi(t) = \frac{1}{p}t^p$, we write $\bV_p(\cdot) \equiv \bV_\varphi(\cdot)$, that is,
\begin{equation}\label{eq: 2.13}
    \bV_p(\bz) := |\bz|^{\frac{p-2}{2}}\bz.
\end{equation}
Similarly, for $\varphi(t) = \frac{1}{p}t^p \log^{\gamma}(e+t)$ with $\gamma>0$, we write $\bV_{{\log}^\gamma}(\cdot) \equiv \bV_\varphi(\cdot)$, that is,
\begin{equation}\label{eq: 2.14}
    \bV_{{\log}^\gamma}(\bz) := \left( |\bz|^{p-2}\log^{\gamma}(e+|\bz|) + \frac{\gamma|\bz|^{p-1}\log^{\gamma-1}(e+|\bz|)}{p(e+|\bz|)} \right)^{\frac{1}{2}}\bz.
\end{equation}
For the $p$-Laplacian case we also recall the following classical inequalities, see for example \cite{MR1895714, MR1962933}.
\begin{lemma}    \label{lem: 2.7}
    Let $\bV_p: \mr^{N\times n} \to \mr^{N\times n}$ be defined by \eqref{eq: 2.13}. Then for all $\bz_1,\bz_2 \in \mr^{N\times n}$,
    \begin{equation}    \label{eq: 2.15}
        \frac{1}{c}(|\bz_1|+|\bz_2|)^{\frac{p-2}{2}}|\bz_1-\bz_2| \leq |\bV_p(\bz_1)-\bV_p(\bz_2)| \leq c(|\bz_1|+|\bz_2|)^{\frac{p-2}{2}}|\bz_1-\bz_2|,
    \end{equation}
    where $c\equiv c(n,p)$. In particular, when $p\geq 2$, we have
    \begin{equation}    \label{eq: 2.16}
        |\bz_1-\bz_2|^p\leq c|\bV_p(\bz_1)-\bV_p(\bz_2)|^2,
    \end{equation}
    for a constant $c\equiv c(n,p)$.
\end{lemma}

\begin{remark}    \label{rmk: 2.8}
    If $\varphi_1$ and $\varphi_2$ are $N$-functions satisfying Assumption \ref{asum: 2.3} with constants $c_{\varphi_1}$ and $c_{\varphi_2}$, respectively, then their sum $\varphi := \varphi_1 + \varphi_2$ satisfies Assumption \ref{asum: 2.3} with the constant $c_\varphi := 2\max\{c_{\varphi_1}, c_{\varphi_2}\}$. Indeed,
    \begin{equation*}
        \frac{\varphi_1''(t)t}{\varphi_1'(t)+\varphi_2'(t)}+\frac{\varphi_2''(t)t}{\varphi_1'(t)+\varphi_2'(t)} \leq \frac{\varphi_1''(t)t}{\varphi_1'(t)}+\frac{\varphi_2''(t)t}{\varphi_2'(t)}\leq c_{\varphi_1}+c_{\varphi_2}
    \end{equation*}
    and
    \begin{equation*}
        \frac{\varphi_1''(t)t}{\varphi_1'(t)+\varphi_2'(t)}+\frac{\varphi_2''(t)t}{\varphi_1'(t)+\varphi_2'(t)} \geq \frac{[\varphi_1''(t)+\varphi_2''(t)]t}{2\max\left\{ \varphi_1'(t),\varphi_2'(t) \right\}} \geq \frac{1}{2} \min\left\{ \frac{1}{c_{\varphi_1}},\frac{1}{c_{\varphi_2}} \right\}.
    \end{equation*}
\end{remark}
\begin{remark}\label{rmk: 2.9}
    The function
    \begin{equation}\label{eq: 2.17}
        \varphi(t) := t^p + a_0 t^p \log^\gamma(e+t), \quad p > 1, \; a_0 \geq 0, \text{ and } \gamma > 0,
    \end{equation}
    is an $N$-function, and by Remark~\ref{rmk: 2.8} it satisfies Assumption~\ref{asum: 2.3} with a constant $c_\varphi \equiv c_\varphi(p,\gamma)$ independent of $a_0$. The regularity results of \cite{diening2009everywhere} for autonomous functionals with Orlicz growth, which we use in Subsection~\ref{subsection 2.5}, require in addition to \eqref{eq: 2.8} a H\"older type condition on $\varphi''$, namely that
    \begin{equation*}
        |\varphi''(t+s)-\varphi''(t)| \leq c\,\varphi''(t)\left( \frac{|s|}{t} \right)^{\beta} \qquad \text{whenever } t>0 \text{ and } |s|<\frac{t}{2},
    \end{equation*}
    for some $\beta\in(0,1]$ and $c\geq1$, see \cite[Assumption~2.2]{diening2009everywhere}. The function in \eqref{eq: 2.17} satisfies this condition with $\beta=1$ and with $c\equiv c(p,\gamma)$, again independently of $a_0$. Indeed, writing $\psi(t):=t^p\log^\gamma(e+t)$, a direct computation gives
    \begin{gather*}
        \frac{1}{c}\,t^{p-2}\log^\gamma(e+t) \leq \psi''(t) \leq c\,t^{p-2}\log^\gamma(e+t), \\
        t|\psi'''(t)| \leq c\,t^{p-2}\log^\gamma(e+t),
    \end{gather*}
    for every $t>0$, where $c\equiv c(p,\gamma)\geq1$, while $t|(t^p)'''|=|p-2|(t^p)''$. Hence $t|\varphi'''(t)|\leq c(p,\gamma)\varphi''(t)$ for every $t>0$. Since \eqref{eq: 2.8} implies $\varphi'(\lambda t)\leq \lambda^{c_\varphi}\varphi'(t)$ for every $\lambda\geq1$ and $t>0$, we also have $\varphi''(\tau)\approx\varphi''(t)$ whenever $t/2\leq\tau\leq 3t/2$, with implicit constants depending only on $c_\varphi$, and the above condition with $\beta=1$ follows from the mean value theorem.
\end{remark}

The following Sobolev--Poincar\'e type inequality follows from \cite[Theorem~7]{diening2008fractional}.
\begin{lemma}    \label{lem: 2.10} 
    Let $\varphi:[0,\infty) \to [0,\infty)$ be an $N$-function satisfying Assumption \ref{asum: 2.3}. Let $B_R \equiv B_R(x_0) \subset \Omega$. Then there exist an exponent $d_1\equiv d_1(n,N,c_\varphi)\in(0,1)$ and a constant $c\equiv c(n,N,c_\varphi)$ such that
    \begin{equation}    \label{eq: 2.18}
        \Xint-_{B_R}\varphi \left( \frac{|\bm{f}-(\bm{f})_{B_R}|}{R} \right)\,dx \leq c\left( \Xint-_{B_R} [\varphi(|D\bm{f}|)]^{d_1}\,dx \right)^{\frac{1}{d_1}}
    \end{equation}
    for every function $\bm{f} \in W^{1,1}(\Omega;\mr^N)$ with $\ds \int_{\Omega}\varphi(|D\bm{f}|)\,dx <\infty$.
\end{lemma}

\subsection{\bf Miscellanea}\label{subsection 2.4} In this subsection, we collect some lemmas and elementary inequalities used throughout the paper. We begin with two classical iteration lemmas. The first one is \cite[Lemma~7.3]{MR1962933}.
\begin{lemma}    \label{lem: 2.11}
    Let $\phi:(0,\tilde{R}] \to [0,\infty)$ be a nondecreasing function and let $\varepsilon\geq0$ and $\tilde{c}\geq1$ be such that
    \begin{equation*}
        \phi(\varrho)\leq \tilde{c}\left[ \left( \frac{\varrho}{R} \right)^n +\varepsilon \right]\phi(R)
    \end{equation*}
    whenever $0 <\varrho \leq R \leq \tilde{R}$. Then for every $\delta \in (0,n)$ there exists $\bar{\varepsilon}\equiv \bar{\varepsilon}(n,\delta, \tilde{c})>0$ such that, if $\varepsilon \leq \bar{\varepsilon}$, then
    \begin{equation*}
        \phi(\varrho)\leq \bar{c}\left( \frac{\varrho}{R} \right)^{n-\delta}\phi(R)
    \end{equation*}
    whenever $0<\varrho \leq R \leq \tilde{R}$, for a constant $\bar{c}\equiv \bar{c}(n,\delta,\tilde{c})$.
\end{lemma}

The second one is \cite[Lemma~6.1]{MR1962933}.
\begin{lemma}    \label{lem: 2.12}
    Let $h:[r_1,r_2]\to [0,\infty)$ be a bounded function such that
    \begin{equation*}
        h(r)\leq \theta h(s)+\frac{A}{(s-r)^\ell} \quad \text{for every } r_1 \leq r<s\leq r_2,
    \end{equation*}
    where $\theta \in (0,1)$, $A \geq 0$ and $\ell>0$. Then
    \begin{equation*}
        h(r)\leq \frac{c(\theta,\ell)A}{(r_2-r_1)^{\ell}}.
    \end{equation*}
\end{lemma}

The following elementary properties of the logarithm will be used frequently:
\begin{equation}    \label{eq: 2.19}
    \log^{\gamma}(e+xy) \leq c_\gamma\left[ \log^{\gamma}(e+x)+\log^{\gamma}(e+y) \right], \qquad c_\gamma:=\max\{1,2^{\gamma-1}\},
\end{equation}
and
\begin{equation}    \label{eq: 2.20}
    \log^{\gamma}(e+Ax) \leq A^{\gamma}\log^{\gamma}(e+x)
\end{equation}
for every $x,y \geq 0$, $A \geq 1$ and $\gamma>0$, where \eqref{eq: 2.19} follows from $e+xy\leq(e+x)(e+y)$ and \eqref{eq: 2.20} from $e+Ax\leq(e+x)^A$. Also,
\begin{equation}    \label{eq: 2.21}
    (x+y)^p\log^{\gamma}(e+x+y)\leq cx^p\log^{\gamma}(e+x)+cy^p\log^{\gamma}(e+y)
\end{equation}
for all $x,y \geq 0$ and all $p\geq 1$, where $c=2^p(\log2+1)^\gamma$.

\subsection{\bf A priori estimates for frozen functionals}\label{subsection 2.5}
In this subsection, we gather regularity results for vector-valued minimizers of the frozen functionals associated with \eqref{eq: 1.2}, that is, with a constant coefficient. These estimates are consequences of the results in \cite{diening2009everywhere}. We consider functionals of the form
\begin{equation}    \label{eq: 2.22}
    \mc{P}_0(\bw, \Omega) := \int_{\Omega} \left( |D\bw|^p + a_0 |D\bw|^p \log^{\gamma}(e + |D\bw|) \right) dx,
\end{equation}
where $a_0 \ge 0$ is a constant. The following higher-order regularity result follows from \cite{diening2009everywhere}.
\begin{theorem}    \label{thm: 2.13}
    Let $\bv \in W^{1,p}(\Omega;\mr^N)$ be a vector-valued local minimizer of the functional $\mc{P}_0$ defined in \eqref{eq: 2.22}. Then $D\bv \in C_{\mathrm{loc}}^{0,\tilde{\alpha}}(\Omega; \mr^{N\times n})$ for some exponent $\tilde{\alpha} \in (0,1)$ depending only on $n$, $N$, $p$ and $\gamma$, and in particular independent of $a_0 \geq 0$ and of $\bv$. Moreover, for every ball $B_R \subset \Omega$, the following inequalities hold:
    \begin{equation}    \label{eq: 2.23}
        \sup_{B_{R/2}} (|D\bv|^p+a_0|D\bv|^p\log^{\gamma}(e+|D\bv|)) \leq c \Xint-_{B_R} (|D\bv|^p+a_0|D\bv|^p\log^{\gamma}(e+|D\bv|)) \,dx
    \end{equation}
    and, for every $0 < \varrho \leq R$,
    \begin{equation}    \label{eq: 2.24}
        \Xint-_{B_\varrho}|D\bv-(D\bv)_{B_\varrho}|^p \,dx \leq c\left( \frac{\varrho}{R} \right)^{\tilde{\alpha}p}\Xint-_{B_R}(|D\bv|^p+a_0|D\bv|^p\log^{\gamma}(e+|D\bv|))\,dx,
    \end{equation}
    where $c$ depends only on $n,N,p$ and $\gamma$.
\end{theorem}

\begin{proof}
    The result follows from \cite{diening2009everywhere}, applied with the choice
    \begin{equation*}
        \varphi(t):= t^p+a_0t^p\log^\gamma(e+t)
    \end{equation*}
    and with the vector field $\bV_\varphi$ defined in \eqref{eq: 2.7}. By Remark~\ref{rmk: 2.9}, $\varphi$ satisfies \cite[Assumption~2.2]{diening2009everywhere} with constants depending only on $p$ and $\gamma$, and this is the reason why the exponent $\tilde{\alpha}$ and the constants below are independent of $a_0$. Estimate \eqref{eq: 2.23} is then contained in \cite{diening2009everywhere}, and it remains to prove \eqref{eq: 2.24}. We first recall the excess decay estimate
    \begin{equation}    \label{eq: 2.25}
        \Xint-_{B_\varrho}|\bV_\varphi(D\bv)-(\bV_\varphi(D\bv))_{B_\varrho}|^2\,dx \leq c\left( \frac{\varrho}{R} \right)^{\alpha_0} \Xint-_{B_R}|\bV_\varphi(D\bv)-(\bV_\varphi(D\bv))_{B_R}|^2\,dx
    \end{equation}
    valid for every $\varrho \in (0,R)$, where $c\geq 1$ and $\alpha_0 \in (0,1)$ depend only on $n$, $N$, $p$ and $\gamma$, see \cite[Lemma~6.4]{diening2009everywhere}. 
    
    It suffices to prove \eqref{eq: 2.24} for $\varrho \leq R/2$. Since $\bV_\varphi$ is a bijection of $\mr^{N\times n}$, there exists $\mathbf{A} \in \mr^{N\times n}$ such that
    \begin{equation}    \label{eq: 2.26}
        \bV_\varphi(\bA)=(\bV_\varphi(D\bv))_{B_\varrho},
    \end{equation}
    that is, $\bA=\bV_\varphi^{-1}(\bV_\varphi(D\bv))_{B_{\varrho}}$. 
    
    We consider two cases. If $p\geq 2$, then \eqref{eq: 2.11}, \eqref{eq: 2.12} and \eqref{eq: 2.26}, together with the definition of $\bV_\varphi$ and a few elementary computations, give
    \begin{align*}
        & \Xint-_{B_\varrho}|D\bv-(D\bv)_{B_\varrho}|^p \,dx \leq c\Xint-_{B_\varrho}|D\bv-\bA|^p \,dx \\
        & \quad \leq c\Xint-_{B_\varrho}(|D\bv|+|\bA|)^{p-2}|D\bv-\bA|^2 \,dx \leq c\Xint-_{B_\varrho}\varphi''(|D\bv|+|\bA|)|D\bv-\bA|^2\,dx\\
        & \quad \leq c\Xint-_{B_\varrho}|\bV_{\varphi}(D\bv)-\bV_\varphi(\bA)|^2 \,dx = c\Xint-_{B_\varrho}|\bV_{\varphi}(D\bv)-(\bV_{\varphi}(D\bv))_{B_\varrho}|^2 \,dx\\
        & \quad \leq c\left( \frac{\varrho}{R} \right)^{\alpha_0}\Xint-_{B_R}|\bV_{\varphi}(D\bv)-(\bV_{\varphi}(D\bv))_{B_R}|^2 \,dx\\
        & \quad \leq c\left( \frac{\varrho}{R} \right)^{\alpha_0}\Xint-_{B_R}\left( |\bV_{\varphi}(D\bv)|^2+|(\bV_{\varphi}(D\bv))_{B_R}|^2 \right) \,dx\\
        & \quad \leq c\left( \frac{\varrho}{R} \right)^{\alpha_0}\Xint-_{B_R}(|D\bv|^p+a_0|D\bv|^p\log^{\gamma}(e+|D\bv|)) \,dx,
    \end{align*}
    where $c\equiv c(n,N,p,\gamma)$. Therefore, \eqref{eq: 2.24} holds with $\tilde{\alpha}:=\alpha_0/p$ if $p\geq 2$. If $1<p<2$, we argue slightly differently. By the triangle inequality and H\"older's inequality, we obtain
    \begin{align*}
        & \Xint-_{B_\varrho} |D\bv - (D\bv)_{B_\varrho}|^p \, dx \\
        & \leq 2^p \Xint-_{B_\varrho} (|D\bv| + |\bA|)^{\frac{p(p-2)}{2}} |D\bv - \bA|^p (|D\bv| + |\bA|)^{\frac{p(2-p)}{2}} \, dx \\
        & \leq c \left( \Xint-_{B_\varrho} (|D\bv| + |\bA|)^{p-2} |D\bv - \bA|^2 \, dx \right)^{\frac{p}{2}} \left( \Xint-_{B_\varrho} (|D\bv| + |\bA|)^p \, dx \right)^{\frac{2-p}{2}} \\
        & \leq c \left( \Xint-_{B_\varrho} \varphi''(|D\bv| + |\bA|) |D\bv - \bA|^2 \, dx \right)^{\frac{p}{2}} \left( \Xint-_{B_\varrho} (|\bV_\varphi(D\bv)|^2 + |\bV_\varphi(\bA)|^2) \, dx \right)^{\frac{2-p}{2}} \\
        & \leq c \left( \Xint-_{B_\varrho} |\bV_\varphi(D\bv) - \bV_\varphi(\bA)|^2 \, dx \right)^{\frac{p}{2}} \left( \Xint-_{B_\varrho} (|\bV_\varphi(D\bv)|^2 + |(\bV_\varphi(D\bv))_{B_\varrho}|^2) \, dx \right)^{\frac{2-p}{2}} \\
        & \leq c \left( \Xint-_{B_\varrho} |\bV_\varphi(D\bv) - (\bV_\varphi(D\bv))_{B_\varrho}|^2 \, dx \right)^{\frac{p}{2}} \sup_{B_\varrho} |\bV_\varphi(D\bv)|^{2-p} \\
        & \leq c \left( \frac{\varrho}{R} \right)^{\frac{\alpha_0 p}{2}} \left( \Xint-_{B_R} |\bV_\varphi(D\bv) - (\bV_\varphi(D\bv))_{B_R}|^2 \, dx \right)^{\frac{p}{2}} \left( \sup_{B_\varrho} \varphi(|D\bv|) \right)^{\frac{2-p}{2}}\\
        & \leq c \left( \frac{\varrho}{R} \right)^{\frac{\alpha_0 p}{2}} \left( \Xint-_{B_R} (|\bV_\varphi(D\bv)|^2 + |(\bV_\varphi(D\bv))_{B_R}|^2) \, dx \right)^{\frac{p}{2}}\\
        & \qquad \times \left( \Xint-_{B_R} (|D\bv|^p + a_0|D\bv|^p \log^\gamma(e+|D\bv|)) \, dx \right)^{\frac{2-p}{2}} \\
        & \leq c \left( \frac{\varrho}{R} \right)^{\frac{\alpha_0 p}{2}} \Xint-_{B_R} (|D\bv|^p + a_0|D\bv|^p \log^\gamma(e+|D\bv|)) \, dx.
    \end{align*}
    Therefore, \eqref{eq: 2.24} holds with $\tilde{\alpha}:=\alpha_0/2$ if $1<p<2$, and the proof is complete.
\end{proof}

\begin{lemma}    \label{thm: 2.14}
    There exists $\delta_0\equiv\delta_0(n,N,p)\in(0,1)$ with the following property. Let $B_{R/2}\subset\Omega$ be a ball, let $\bu\in W^{1,p(1+\delta)}(B_{R/2};\mr^N)$ for some $\delta\in(0,\delta_0]$, and let $\bv\in\bu+W_0^{1,p}(B_{R/2};\mr^N)$ be the minimizer of the functional
    \begin{equation*}
        \bw \mapsto \int_{B_{R/2}} |D\bw|^p\,dx, \qquad p>1,
    \end{equation*}
    in the Dirichlet class $\bu+W_0^{1,p}(B_{R/2};\mr^N)$. Then $D\bv \in L^{p(1+\delta)}(B_{R/2}; \mr^{N\times n})$ and
    \begin{equation*}
        \Xint-_{B_{R/2}}|D\bv|^{p(1+\delta)} \,dx \leq c \Xint-_{B_{R/2}}|D\bu|^{p(1+\delta)}\,dx,
    \end{equation*}
    where $c\equiv c(n,N,p)$.
\end{lemma}
This is a classical consequence of a Caccioppoli type inequality up to the boundary for $\bv-\bu$, of the global version of Gehring's lemma and of the minimality of $\bv$, see \cite{giaquinta1979regularity} and \cite[Chapter~V]{giaquinta1983multiple}.

\section{Gehring's theory} \label{section 3}
We now turn to lower-order regularity results. The model functional $\plog$ in \eqref{eq: 1.2} enters the discussion only through the growth conditions of its integrand. Thus, we consider functionals of the form
\begin{equation}\label{eq: 3.1}
W^{1,1}(\Omega;\mr^N) \ni \bw \mapsto \mc{F}(\bw,\Omega)
:= \int_{\Omega}F(x,\bw,D\bw) \,dx,
\end{equation}
where $F:\Omega \times \mr^N \times \mr^{N \times n} \to \mr$ is a Carath\'eodory function satisfying the growth condition
\begin{equation}    \label{eq: 3.2}
    \nu H(x,\bz) \leq F(x,\bv, \bz) \leq L H(x,\bz)
\end{equation}
for all $x \in \Omega$, $\bv \in \mr^N$ and $\bz \in \mr^{N\times n}$, with $0 < \nu \leq 1 \leq L$. Here $H: \Omega \times \mr^{N \times n} \to \mr$ is defined by
\begin{equation}    \label{eq: 3.3}
    H(x,\bz)=|\bz|^p+a(x)|\bz|^p\log^\gamma(e+|\bz|).
\end{equation}
With a slight abuse of notation, we use $H(x,\bz)$ both for $\bz\in\mr^N$ and for $\bz\in\mr^{N\times n}$, the meaning being clear from the context.

All the results of this and of the next section are local. Since a local minimizer $\bu$ of $\plog$ satisfies $H(\cdot,D\bu)\in L^1_{\mathrm{loc}}(\Omega)$, and hence $\bu\in W^{1,p}_{\mathrm{loc}}(\Omega;\mr^N)$, after replacing $\Omega$ by an open subset $\Omega_0\Subset\Omega$ we may and shall assume in the following that $\bu\in W^{1,p}(\Omega;\mr^N)$ and $H(\cdot,D\bu)\in L^1(\Omega)$. Accordingly, $\|D\bu\|_{L^p}$ stands for $\|D\bu\|_{L^p(\Omega)}$ whenever it appears in the dependence of the constants.

Let $\omega(\cdot)$ denote the modulus of continuity of $a(\cdot)$ as in \eqref{eq: 1.6}. Without loss of generality, we assume that $\omega(\cdot)$ is concave. The natural condition for the basic regularity is
\begin{equation}\label{eq: 3.4}
\limsup_{r\to 0} \omega(r)\log^\gamma\left(\frac{1}{r}\right)<\infty,
\end{equation}
which can be equivalently expressed as
\begin{equation}\label{eq: 3.5}
\omega(r)\log^\gamma\left(\frac{1}{r}\right)\leq\tilde{L}
\end{equation}
for some constant $\tilde{L}\geq1$ and every $0<r\leq 1$.

The following lemma expresses the self-improving property of reverse H\"older inequalities, see \cite[\S 6.4]{MR1962933}.
\begin{lemma}  \label{lem: 3.1}
    Assume that a function $f \in L^1_{\mathrm{loc}}(\Omega)$ satisfies
    \begin{equation*}
        \Xint-_{B_{R/2}} |f| \,dx \leq \tilde{c}\left( \Xint-_{B_R} |f|^d \,dx \right)^{\frac{1}{d}}
    \end{equation*}
    for some exponent $d \in (0,1)$, some constant $\tilde{c}\geq 1$ and every ball $B_R \subset \Omega$ with radius $R \leq R_0$. Then there exists an exponent $\delta_g\equiv\delta_g(n,d,\tilde{c})>0$ such that $f \in L_{\mathrm{loc}}^{1+\delta_g}(\Omega)$ and
    \begin{equation*}
        \Xint-_{B_{R/2}}|f|^{1+\delta_g}\,dx \leq c(n,d,\tilde{c})\left( \Xint-_{B_R}|f|\,dx \right)^{1+\delta_g}
    \end{equation*}
    for every ball $B_R \subset \Omega$ with radius $R \leq R_0$.
\end{lemma}
We also recall the Poincar\'e inequality on balls, see \cite[\S 5.8.1, Theorem~2]{MR1625845}.
\begin{lemma}  \label{lem: 3.2}
    Let $1\leq p\leq \infty$, let $B_R \subset \mr^n$ be a ball and let $\bu \in W^{1,p}(B_R;\mr^N)$. Then there exists a constant $c$ depending only on $n$, $N$ and $p$ such that
    \begin{equation} \label{eq: 3.6}
        \int_{B_R}|\bu -(\bu)_{B_R}|^p\,dx\leq cR^p\int_{B_R}|D\bu|^p\,dx.
    \end{equation}
\end{lemma}

The main result of this section is the following higher integrability theorem.
\begin{theorem}    \label{thm: 3.3}
    Let $\bu \in W^{1,p}(\Omega;\mr^N)$ be a vector-valued local minimizer of the functional $\mathcal{F}$ defined in \eqref{eq: 3.1}, and assume that the modulus of continuity $\omega(\cdot)$ of $a(\cdot)$ satisfies \eqref{eq: 3.5}. Then there exists $\delta_g>0$, depending only on $n$, $N$, $p$, $\nu$, $L$, $\|D\bu\|_{L^p}$, $\tilde{L}$ and $\gamma$, such that
    \begin{equation} \label{eq: 3.7}
        H(x,D\bu) \in L^{1+\delta_g}_{\mathrm{loc}}(\Omega).
    \end{equation}
    More precisely, for every ball $B_R\subset\Omega$ with $R\leq 1/e$, we have the local reverse H\"older inequality
    \begin{equation} \label{eq: 3.8}
        \left( \Xint-_{B_{R/2}} [H(x,D\bu)]^{1+\delta_g} \, dx \right)^{\frac{1}{1+\delta_g}} \leq c \Xint-_{B_R} H(x,D\bu) \, dx,
    \end{equation}
    where $c$ depends only on $n$, $N$, $p$, $\nu$, $L$, $\|D\bu\|_{L^p}$, $\tilde{L}$ and $\gamma$. In particular, $\bu$ is locally H\"older continuous provided that $p>n/(1+\delta_g)$.
\end{theorem}

We begin with the following Caccioppoli type inequality. Note that the boundedness of $a(\cdot)$ is the only assumption needed here, and no continuity of $a(\cdot)$ is required.
\begin{lemma}    \label{lem: 3.4}
    Let $\bu \in W^{1,p}(\Omega; \mr^N)$ be a vector-valued local minimizer of the functional $\mc{F}$ defined in \eqref{eq: 3.1}, where $a(\cdot)$ is bounded and nonnegative. Then the inequality
    \begin{equation}    \label{eq: 3.9}
        \Xint-_{B_{R/2}} H(x,D\bu)\,dx \leq c \Xint-_{B_R}H\left( x, \frac{\bu-(\bu)_{B_R}}{R} \right)\,dx
    \end{equation}
    holds for every ball $B_R \subset \Omega$, with a constant $c \equiv c(n,N,p,\nu,L,\gamma)$.
\end{lemma}

\begin{proof}
    Fix $r_1,r_2$ with $0<r_1<r_2\leq R$ and $r,s$ with $r_1\leq r<s\leq r_2$. Choose a cut-off function $\eta\in C_c^\infty(B_s)$ such that $0\leq\eta\leq1$, $\eta\equiv1$ on $B_r$ and $|D\eta|\leq c/(s-r)$, and set $\bv:=\bu-\eta(\bu-(\bu)_{B_R})$. By the minimality of $\bu$ and \eqref{eq: 3.2}, we obtain
    \begin{equation*}
        \int_{B_s} H(x,D\bu)\,dx \leq \frac{L}{\nu}\int_{B_s}H(x,D\bv)\,dx.
    \end{equation*}
    On $B_s \backslash B_r$, a direct computation based on \eqref{eq: 2.21} yields
    \begin{align*}
        H(x,D\bv)\leq cH(x,D\bu)+cH\left( x,\frac{c(\bu-(\bu)_{B_R})}{s-r} \right).
    \end{align*}
    Since $D\bv=0$ on $B_r$, we get
    $$
    \int_{B_r}H(x,D\bu)\,dx \leq c_*\int_{B_s\backslash B_r}H(x,D\bu)\,dx + c_*\int_{B_s}H\left( x,\frac{c(\bu-(\bu)_{B_R})}{s-r} \right)\,dx
    $$
    for a constant $c_*\equiv c_*(n,N, p,\gamma,\nu,L)$. Adding $c_*\int_{B_r}H(x,D\bu)\,dx$ to both sides, dividing by $c_*+1$ and using \eqref{eq: 2.20}, we obtain
    \begin{align*}
        & \int_{B_r}H(x,D\bu)\,dx \\
        & \leq \theta \int_{B_s} H(x,D\bu)\,dx + c \int_{B_s}H\left( x, \frac{\bu-(\bu)_{B_R}}{s-r} \right)\,dx\\
        &= \theta \int_{B_s} H(x,D\bu)\,dx \\
        & \quad +\frac{c}{|s-r|^p} \int_{B_s}|\bu-(\bu)_{B_R}|^p\left( 1+a(x)\log^\gamma \left( e+\left( \frac{r_2-r_1}{s-r} \right)\Big|\frac{\bu-(\bu)_{B_R}}{r_2-r_1}\Big| \right) \right)\,dx \\
        & \leq \theta \int_{B_s} H(x,D\bu)\,dx\\
        & \quad + c\left( \frac{r_2-r_1}{s-r} \right)^{p+\gamma} \int_{B_s}\Big|\frac{\bu-(\bu)_{B_R}}{r_2-r_1}\Big|^p\left( 1+a(x)\log^\gamma \left( e+\Big|\frac{\bu-(\bu)_{B_R}}{r_2-r_1}\Big| \right) \right)\,dx
    \end{align*}
    where $\theta:=c_*/(c_*+1) \in (0,1)$ and $c$ depends only on $n$, $N$, $p$, $\gamma$, $\nu$ and $L$. Finally, applying Lemma~\ref{lem: 2.12} with $\ell=p+\gamma$ to the function
    \begin{equation*}
        h(t)\equiv \int_{B_t}H(x,D\bu)\,dx, \qquad t\in[R/2,R],
    \end{equation*}
    and using \eqref{eq: 2.20} once more, we obtain \eqref{eq: 3.9}.
\end{proof}

\begin{proof}[Proof of Theorem~\ref{thm: 3.3}]
    It suffices to show that there exists an exponent $d \in (0,1)$, depending only on $n,N,p,\nu,L$ and $\gamma$, such that for every ball $B_R \subset \Omega$ with $R \leq 1/e$ we have
    \begin{equation}    \label{eq: 3.10}
        \Xint-_{B_{R/2}} H(x, D\bu)\,dx \leq c \left( \Xint-_{B_R}[H(x,D\bu)]^d \,dx\right)^{\frac{1}{d}},
    \end{equation}
    where $c$ depends only on $n,N,p,\nu,L,\|D\bu\|_{L^p},\tilde{L}$ and $\gamma$. Indeed, \eqref{eq: 3.8} then follows from Lemma~\ref{lem: 3.1} with $R_0=1/e$, and \eqref{eq: 3.7} follows by a standard covering argument. We distinguish two cases according to whether or not the inequality
    \begin{equation}    \label{eq: 3.11}
        a_i(R):= \inf_{B_R}a \leq 2 \omega(R)
    \end{equation}
    holds. We start with the case in which \eqref{eq: 3.11} holds. Then 
    \begin{equation*}
        a(x)=[a(x)-a_i(R)]+a_i(R) \leq \omega(2R)+a_i(R) \leq 2\omega(R)+a_i(R)\leq 4\omega(R)
    \end{equation*}
    for every $x \in B_R$, where we used the concavity of $\omega(\cdot)$ to obtain $\omega(2R)\leq2\omega(R)$. The same argument will be used repeatedly below. Using \eqref{eq: 2.19}, \eqref{eq: 2.20}, Lemma~\ref{lem: 3.2} and $R\leq 1/e$, we obtain
    \begin{equation}    \label{eq: 3.12}
        \begin{aligned}
            & \log^{\gamma}\left( e+\Big|\frac{\bu-(\bu)_{B_R}}{R}\Big| \right)\\
            & \quad \leq c\left[ 1+\log^\gamma \left( e+ \Big| \frac{\bu-(\bu)_{B_R}}{R} \Big|^p \right) \right]\\
            & \quad = c+c\log^\gamma \left( e+ \frac{|\bu-(\bu)_{B_R}|^p}{(|\bu-(\bu)_{B_R}|^p)_{B_R}}\cdot \frac{(|\bu-(\bu)_{B_R}|^p)_{B_R}}{R^p} \right)\\
            & \quad \leq c+c\log^{\gamma}\left( e+\frac{|\bu-(\bu)_{B_R}|^p}{(|\bu-(\bu)_{B_R}|^p)_{B_R}} \right)+c\log^{\gamma}\left( e+\frac{(|\bu-(\bu)_{B_R}|^p)_{B_R}}{R^p}  \right)\\
            & \quad \leq c+c\log^{\gamma}\left( e+\frac{|\bu-(\bu)_{B_R}|^p}{(|\bu-(\bu)_{B_R}|^p)_{B_R}} \right)+c\log^\gamma\left( e+\frac{\|D\bu\|_{L^p(B_R)}^p}{R^n} \right)\\
            & \quad \leq c\log^{\gamma}\left( e+\frac{|\bu-(\bu)_{B_R}|^p}{(|\bu-(\bu)_{B_R}|^p)_{B_R}} \right)+c\log^\gamma\left(\frac{1}{R}\right)
        \end{aligned}
    \end{equation}
    for some $c\equiv c(n,N,p,\|D\bu\|_{L^p},\gamma)$. Combining this inequality with \eqref{eq: 3.9} and with \eqref{eq: 2.5}, the latter applied with $\varrho=R$, $f \equiv |\bu-(\bu)_{B_R}|/R$ and $q/p$ in place of $q$, we deduce that for every $q>p$
    \begin{align*}
        & \Xint-_{B_{R/2}} H(x,D\bu)\,dx \\
        & \quad\leq c\Xint-_{B_R} H\left( x, \frac{\bu-(\bu)_{B_R}}{R} \right)\,dx \\
        & \quad \leq c \Xint-_{B_R} \Big|\frac{\bu-(\bu)_{B_R}}{R}\Big|^p +c \omega(R)\Big|\frac{\bu-(\bu)_{B_R}}{R}\Big|^p \log^{\gamma}\left( e+\Big|\frac{\bu-(\bu)_{B_R}}{R}\Big| \right)\,dx \\
        & \quad\leq c\left( 1+\omega(R)\log^\gamma \left( \frac{1}{R} \right) \right)\Xint-_{B_R}\Big|\frac{\bu-(\bu)_{B_R}}{R}\Big|^p\,dx\\
        & \quad \quad +c\omega(R)\Xint-_{B_R}\Big|\frac{\bu-(\bu)_{B_R}}{R}\Big|^p\log^\gamma\left( e+\frac{|\bu-(\bu)_{B_R}|^p}{(|\bu-(\bu)_{B_R}|^p)_{B_R}} \right)\,dx\\
        & \quad \leq c\left( 1+\omega(R)\log^\gamma \left( \frac{1}{R} \right) \right)\left[ \Xint-_{B_R} \Big|\frac{\bu-(\bu)_{B_R}}{R}\Big|^p \,dx +\left( \Xint-_{B_R}\Big|\frac{\bu-(\bu)_{B_R}}{R}\Big|^q\,dx\right)^{\frac{p}{q}}\right]
    \end{align*}
    for a constant $c$ depending only on $n,N,p,\nu,L,\|D\bu\|_{L^p}$ and $\gamma$. By the Sobolev embedding theorem, we can choose exponents $q_*$ and $q$ with $q_*<p<q$ such that
    \begin{equation*}
        \left( \Xint-_{B_R}\Big|\frac{\bu-(\bu)_{B_R}}{R}\Big|^q\,dx \right)^{\frac{p}{q}} \leq c\left( \Xint-_{B_R}|D\bu|^{q_*}\,dx\right)^{\frac{p}{q_*}}.
    \end{equation*}
    Combining this with \eqref{eq: 3.5} and H\"older's inequality, we obtain
    \begin{align*}
        & \Xint-_{B_{R/2}} H(x,D\bu)\,dx \\
        & \quad \leq c\Xint-_{B_R} \Big|\frac{\bu-(\bu)_{B_R}}{R}\Big|^p \,dx +c\left( \Xint-_{B_R}\Big|\frac{\bu-(\bu)_{B_R}}{R}\Big|^q\,dx\right)^{\frac{p}{q}}\\
        & \quad \leq c\left( \Xint-_{B_R}\Big|\frac{\bu-(\bu)_{B_R}}{R}\Big|^q\,dx\right)^{\frac{p}{q}}\\
        & \quad \leq c\left( \Xint-_{B_R}\Big|\frac{\bu-(\bu)_{B_R}}{R}\Big|^{q_*}\,dx\right)^{\frac{p}{q_*}}\\
        & \quad \leq c\left( \Xint-_{B_R}[H(x,D\bu)]^{\frac{q_*}{p}}\,dx\right)^{\frac{p}{q_*}}
    \end{align*}
    for a constant $c$ depending only on $n,N,p,\nu,L,\|D\bu\|_{L^p},\tilde{L}$ and $\gamma$.

    We next consider the case in which \eqref{eq: 3.11} fails, namely
    \begin{equation*}
        a_i(R)\equiv \inf_{B_R}a > 2\omega(R).
    \end{equation*}
    Then
    \begin{equation*}
        a(x)=[a(x)-a_i(R)]+a_i(R)\leq \omega(2R)+a_i(R) \leq 2\omega(R)+a_i(R)\leq 2a_i(R)
    \end{equation*}
    for every $x\in B_R$, and it follows from \eqref{eq: 3.9} that
    \begin{equation*}
        \Xint-_{B_{R/2}} H(x,D\bu)\,dx \leq c\Xint-_{B_R}\Big|\frac{\bu-(\bu)_{B_R}}{R}\Big|^p\left( 1+a_i(R)\log^\gamma\left( e+\Big|\frac{\bu-(\bu)_{B_R}}{R}\Big| \right) \right)\,dx.
    \end{equation*}
    Finally, we apply the Sobolev--Poincar\'e type inequality \eqref{eq: 2.18} with the choice 
    $$
    \varphi(t):=t^p+a_0t^p\log^{\gamma}(e+t),
    $$
    as in Remark~\ref{rmk: 2.9}. This yields
    \begin{align*}
        & \Xint-_{B_R}\Big|\frac{\bu-(\bu)_{B_R}}{R}\Big|^p\left( 1+a_0\log^\gamma\left( e+\Big|\frac{\bu-(\bu)_{B_R}}{R}\Big| \right) \right)\,dx\\
        & \qquad \qquad \leq c\left( \Xint-_{B_R}|D\bu|^{pd_1}(1+a_0\log^\gamma (e+|D\bu|))^{d_1} \right)^{\frac{1}{d_1}},
    \end{align*}
    where $d_1 \equiv d_1(n,N,p,\gamma) \in (0,1)$ and $c$ depends only on $n$, $N$, $p$ and $\gamma$ by Remark~\ref{rmk: 2.9}. Applying this inequality with $a_0=a_i(R)$, we deduce that
    \begin{align*}
        \Xint-_{B_{R/2}} H(x,D\bu)\,dx 
        & \leq c\left( \Xint-_{B_R}|D\bu|^{pd_1}(1+a_0\log^\gamma (e+|D\bu|))^{d_1} \right)^{\frac{1}{d_1}}\\
        & \leq c\left( \Xint-_{B_R}[H(x,D\bu)]^{d_1}\,dx \right)^{\frac{1}{d_1}}.
    \end{align*}
    This proves \eqref{eq: 3.10} with $d=\max\{d_1,q_*/p\}\in(0,1)$, and the proof is complete.
\end{proof}

\section{Proof of Theorem~\ref{thm: 1.2}}\label{section 4}
In this section, we prove the two assertions of Theorem~\ref{thm: 1.2}. In Subsection~\ref{subsection 4.2}, we first prove that
\begin{equation*}
    \bu \in C_{\mathrm{loc}}^{0,\beta}(\Omega;\mr^N) \quad \text{for every } \beta \in (0,1)
\end{equation*}
under the assumption $\ell=0$, where $\ell$ is defined in \eqref{eq: 1.7}. Then, in Subsection~\ref{subsection 4.3}, we prove that
$$
D\bu \in C_{\mathrm{loc}}^{0,\beta}(\Omega;\mr^{N\times n})
$$
for some $\beta\in(0,1)$ under the assumption that $a(\cdot)$ is H\"older continuous, that is, $\omega(r)\lesssim r^\sigma$ for some $\sigma\in(0,1)$.
 
\subsection{Comparison lemma and decay estimate} \label{subsection 4.1}
In Lemma~\ref{lem: 4.3} below, we estimate the difference between the energy of a minimizer of $\plog$ on a ball $B_{R/2}\equiv B_{R/2}(x_0)\subset\Omega$ and the corresponding frozen energy, obtained by replacing the coefficient $a(\cdot)$ by its infimum. More precisely, the latter energy is given by
\begin{equation}    \label{eq: 4.1}
    \int_{B_{R/2}}(|D\bu|^p+a_i(R)|D\bu|^p\log^\gamma(e+|D\bu|))\,dx,
\end{equation}
where
\begin{equation}    \label{eq: 4.2}
    a_i(R)\equiv a_i(x_0,R)=\inf_{B_R}a(\cdot).
\end{equation}

Before proceeding, we fix some notation and recall the weak form of the Euler--Lagrange equation of $\plog$, together with the class of admissible test functions.
\begin{remark}    \label{rmk: 4.1}
    For the remainder of the paper, we denote by $f,g:\mr^{N\times n}\to\mr$ the functions
    \begin{equation}    \label{eq: 4.3}
        f(\bz):=|\bz|^p \qquad \text{and} \qquad g(\bz):=|\bz|^p\log^\gamma(e+|\bz|).
    \end{equation}
\end{remark}

\begin{remark}    \label{rmk: 4.2}
    Let $\bu$ be a local minimizer of $\plog$ and let $\bm{\varphi} \in W^{1,p}(\Omega;\mr^N)$ be a function with compact support in $\Omega$ such that $D\bm{\varphi} \in L^p\log^\gamma L(\Omega; \mr^{N\times n})$. Then the Euler--Lagrange equation of $\plog$ holds in the weak form
    \begin{equation}    \label{eq: 4.4}
        \int_{\Omega}\langle \partial f(D\bu)+a(x)\partial g(D\bu),D\bm{\varphi} \rangle\,dx =0,
    \end{equation}
    see \cite[\S 5.1]{MR3570955}.
\end{remark}

\begin{lemma} \label{lem: 4.3}
    Let $\bu \in W^{1,p}(\Omega;\mr^N)$ be a vector-valued local minimizer of $\mc{P}_{\log^{\gamma}}$, let $B_R \equiv B_R(x_0) \subset \Omega$ be a ball with radius $R \leq 1/e$, and let $a_i(R)$ be as in \eqref{eq: 4.2}. Moreover, let $\omega(\cdot)$ be as in \eqref{eq: 1.6} and assume that \eqref{eq: 3.5} holds. Then for every $\tilde{\gamma}>0$ there exists a constant $c\equiv c(n,N,p,\|D\bu\|_{L^p},\tilde{L},\tilde{\gamma},\gamma)$ such that
    \begin{equation} \label{eq: 4.5}
        \begin{aligned}                                       
            &\int_{B_{R/2}}[a(x)-a_i(R)]^{\tilde{\gamma}}|D\mathbf{u}|^p \log^{\gamma \tilde{\gamma}}(e+|D\mathbf{u}|)\,dx \\
            &\hspace*{10em} \leq c \left[ \omega(R) \log^{\gamma}\left( \frac{1}{R} \right) \right]^{\tilde{\gamma}} \int_{B_R} H(x,D\mathbf{u})\,dx.
        \end{aligned}
    \end{equation}
\end{lemma}

\begin{proof}
    By \eqref{eq: 3.5}, Theorem~\ref{thm: 3.3} applies to $\bu$ and yields
    $$
    H(x,D\bu) \in L_{\mathrm{loc}}^{1+\delta_g}(\Omega).
    $$
    Here the exponent $\delta_g>0$ depends only on $n,N,p,\|D\bu\|_{L^p},\tilde{L}$ and $\gamma$, since $\nu=L=1$ for the functional $\plog$. Moreover, by \eqref{eq: 3.8} there exists a constant $c$ with the same dependence such that, for every ball $B_R\subset\Omega$ with $R\leq 1/e$,
    \begin{equation}    \label{eq: 4.6}
        \Xint-_{B_{R/2}}|D\bu|^{p+p\delta_g}\,dx \leq \Xint-_{B_{R/2}}[H(x,D\bu)]^{1+\delta_g}\,dx \leq c\left( \Xint-_{B_R}H(x,D\bu) \,dx \right)^{1+\delta_g}.
    \end{equation}
    Using \eqref{eq: 1.6} to estimate the left-hand side of \eqref{eq: 4.5}, we obtain
    \begin{equation*}
        \begin{aligned}
            &\int_{B_{R/2}}[a(x)-a_i(R)]^{\tilde{\gamma}}|D\bu|^p\log^{\gamma\tilde{\gamma}}(e+|D\bu|)\,dx\\
            & \qquad \leq c[\omega(R)]^{\tilde{\gamma}}\int_{B_{R/2}}|D\bu|^p\log^{\gamma\tilde{\gamma}}(e+|D\bu|^p)\,dx.
        \end{aligned}
    \end{equation*}
    To estimate the right-hand side above, we use \eqref{eq: 2.19} and \eqref{eq: 2.5}, the latter applied with $\varrho=R/2$, $q=1+\delta_g$ and with $\gamma\tilde{\gamma}$ in place of $\gamma$, to deduce that
    \begin{align*}
        & \int_{B_{R/2}}|D\bu|^p\log^{\gamma\tilde{\gamma}}(e+|D\bu|^p)\,dx\\
        & \quad \leq c\int_{B_{R/2}} |D\bu|^p\log^{\gamma\tilde{\gamma}}\left( e+\frac{|D\bu|^p}{(|D\bu|^p)_{B_R}} \right)\,dx \\
        & \hspace*{13em}+ c\log^{\gamma\tilde{\gamma}}(e+(|D\bu|^p)_{B_R}) \int_{B_{R/2}}|D\bu|^p\,dx\\
        & \quad \leq cR^n \left( \Xint-_{B_{R/2}} |D\bu|^{p(1+\delta_g)} \right)^{\frac{1}{1+\delta_g}}+c\log^{\gamma \tilde{\gamma}}\left( \frac{1}{R} \right)\int_{B_{R/2}}|D\bu|^p\,dx
    \end{align*}
    for a constant $c$ depending on $n,N,p,\|D\bu\|_{L^p},\gamma$ and $\tilde{\gamma}$. Combining this with the reverse H\"older inequality \eqref{eq: 4.6} and with $R\leq 1/e$, we obtain
    \begin{equation*} 
        \begin{aligned}
            & \int_{B_{R/2}}|D\bu|^p\log^{\gamma\tilde{\gamma}}(e+|D\bu|^p)\,dx\\
            & \quad \leq cR^n \Xint-_{B_R}H(x,D\bu)\,dx +c\log^{\gamma \tilde{\gamma}}\left( \frac{1}{R} \right) \int_{B_{R/2}}H(x,D\bu)\,dx\\
            & \quad \leq c\log^{\gamma \tilde{\gamma}}\left( \frac{1}{R} \right) \int_{B_{R}}H(x,D\bu)\,dx,
        \end{aligned}
    \end{equation*}
    where the constant $c$ depends on $n,N,p,\tilde{L},\|D\bu\|_{L^p},\gamma$ and $\tilde{\gamma}$. Multiplying the last estimate by $[\omega(R)]^{\tilde{\gamma}}$ and recalling the first estimate of the proof, we obtain \eqref{eq: 4.5}.
\end{proof}

The following comparison lemma estimates the difference between a minimizer of $\plog$ and the minimizer of the corresponding frozen functional. We allow the possibility that $a_i(R)=0$. In this case, the frozen functional reduces to the $p$-Dirichlet energy, and only the vector field $\bV_p$ defined in \eqref{eq: 2.13} enters the estimate.

\begin{lemma}    \label{lem: 4.4}
    Let $\bu \in W^{1,p}(\Omega;\mr^N)$ be a vector-valued local minimizer of $\plog$, let $B_R\equiv B_R(x_0) \Subset \Omega$ be a ball with $R \leq 1/e$, and let $a_i(R)$ be as in \eqref{eq: 4.2}. Moreover, let $\omega(\cdot)$ be as in \eqref{eq: 1.6} and assume that \eqref{eq: 3.5} holds. Let $\bv \in W^{1,p}(B_{R/2};\mr^N)$ be the solution of the Dirichlet problem
    \begin{equation} \label{eq: 4.7}
        \begin{cases}
            \bv \mapsto \underset{\bw}{\min} \int_{B_{R/2}}(|D\bw|^p+a_i(R)|D\bw|^p\log^\gamma(e+|D\bw|))\,dx \\
            \bw \in \bu + W_0^{1,p}(B_{R/2};\mr^N).
        \end{cases}
    \end{equation}
    Then there exists a constant $c\equiv c(n,N,p,\gamma, \|D\bu\|_{L^p},\tilde{L})$ such that 
    \begin{equation}  \label{eq: 4.8}
        \begin{aligned}
            &\int_{B_{R/2}}(|\bV_p(D\bu)-\bV_p(D\bv)|^2+a_i(R)|\bV_{\log^\gamma}(D\bu)-\bV_{\log^\gamma}(D\bv)|^2)\,dx \\
            &\hspace*{10em} \leq c \omega(R)\log^{\gamma}\left( \frac{1}{R} \right)\int_{B_R}H(x,D\bu)\,dx.
        \end{aligned}
    \end{equation}
\end{lemma}

\begin{proof}
    We first observe that problem \eqref{eq: 4.7} admits a minimizer, since
    \begin{equation}    \label{eq: 4.9}
        D\bu \in L^{p(1+\delta_g)}(B_R) \subset L^p\log^\gamma L(B_R;\mr^{N\times n})
    \end{equation}
    by \eqref{eq: 3.7} and $B_R \Subset \Omega$. Hence the Dirichlet class $\bu +W_0^{1,p}(B_{R/2};\mr^N)$ contains a function with finite energy. By the minimality of $\bv$ we have
    \begin{equation}    \label{eq: 4.10}
        \begin{aligned}
            &\int_{B_{R/2}}(|D\bv|^p+a_i(R)|D\bv|^p\log^{\gamma}(e+|D\bv|))\,dx \\
            & \qquad \leq \int_{B_{R/2}} (|D\bu|^p+a_i(R)|D\bu|^p\log^{\gamma}(e+|D\bu|))\,dx.
        \end{aligned}
    \end{equation}
    Consequently, we have
    \begin{equation}    \label{eq: 4.11}
        D\bv \in L^p\log^\gamma L(B_{R/2};\mr^{N\times n}).
    \end{equation}
    Indeed, for $a_i(R)>0$ this follows from \eqref{eq: 4.10}, whose right-hand side is finite by \eqref{eq: 4.9}, whereas for $a_i(R)=0$ it follows from Lemma~\ref{thm: 2.14}, applied with $\delta:=\min\{\delta_g,\delta_0\}$, together with \eqref{eq: 4.9}. Since both $\bu$ and $\bv$ are minimizers, we may use their Euler--Lagrange equations, see Remark~\ref{rmk: 4.2}. Recalling the notation in Remark~\ref{rmk: 4.1}, we note that $\bu-\bv \in W_0^{1,p}(B_{R/2};\mr^N)$ is an admissible test function, since 
    $$
    D\bu-D\bv \in L^p\log^\gamma L(B_{R/2};\mr^{N\times n})
    $$
    by \eqref{eq: 4.9} and \eqref{eq: 4.11}. Testing both equations with $\bu-\bv$, we obtain
    \begin{align*}
        \mc{D}_1
        & := \int_{B_{R/2}}\langle \partial f(D\bu)-\partial f(D\bv)+a_i(R)(\partial g(D\bu)-\partial g(D\bv)),D\bu-D\bv \rangle\,dx\\
        & = \int_{B_{R/2}}\langle \partial f(D\bu)+a_i(R)\partial g(D\bu),D\bu-D\bv \rangle\,dx\\
        & = \int_{B_{R/2}}[a_i(R)-a(x)]\langle \partial g(D\bu),D\bu-D\bv \rangle\,dx=:\mc{D}_2,
    \end{align*}
    where we used the Euler--Lagrange equation of $\bv$ in the second line and \eqref{eq: 4.4} in the third one. Then \eqref{eq: 2.10} applied to $\bV_p$ and $\bV_{\log^{\gamma}}$ yields
    \begin{equation}    \label{eq: 4.12}
        \int_{B_{R/2}}(|\bV_p(D\bu)-\bV_p(D\bv)|^2+a_i(R)|\bV_{\log^{\gamma}}(D\bu)-\bV_{\log^{\gamma}}(D\bv)|^2)\,dx \leq c|\mc{D}_2|,
    \end{equation}
    where $c \equiv c(n,N,p,\gamma)$. We estimate the integrand of $\mc{D}_2$ by Young's inequality. For every $A>0$ we have
    \begin{align*}
        & |[a_i(R)-a(x)]\langle \partial g(D\bu),D\bu-D\bv \rangle|\\
        & \qquad \leq |a_i(R)-a(x)| |\partial g(D\bu)|(|D\bu|+|D\bv|)\\
        & \qquad \leq \frac{c}{A^{\frac{1}{p-1}}}|a(x)-a_i(R)|^{p'}|\partial g(D\bu)|^{p'}+cA(|D\bu|+|D\bv|)^p\\
        & \qquad \leq \frac{c}{A^{\frac{1}{p-1}}}|a(x)-a_i(R)|^{p'}|D\bu|^p\log^{\gamma p'}(e+|D\bu|)+cA(|D\bu|+|D\bv|)^p,
    \end{align*}
    where $c \equiv c(n,N,p,\gamma)$.
    Combining this with Lemma~\ref{lem: 4.3} for $\tilde{\gamma}=p'$, the bound $\log^{\gamma}(1/R)\geq 1$, the minimality property \eqref{eq: 4.10} and the inequality $a_i(R)\leq a(x)$ in $B_{R/2}$, we obtain
    \begin{equation}    \label{eq: 4.13}
        \begin{aligned}
            |\mc{D}_2|
            & \leq \int_{B_{R/2}} |[a_i(R)-a(x)]\langle \partial g(D\bu),D\bu-D\bv \rangle| \,dx\\
            & \leq \frac{c}{A^{\frac{1}{p-1}}}\int_{B_{R/2}}|a(x)-a_i(R)|^{p'}|D\bu|^p\log^{\gamma p'}(e+|D\bu|)\,dx\\
            & \qquad\qquad+ cA\int_{B_{R/2}}(|D\bu|+|D\bv|)^p\,dx\\
            & \leq \frac{c}{A^{\frac{1}{p-1}}}\left[ \omega(R)\log^{\gamma}\left( \frac{1}{R} \right) \right]^{p'}\int_{B_R}H(x,D\bu)\,dx \\
            & \qquad \qquad +cA \int_{B_{R/2}} (|D\bu|^p+a_i(R)|D\bu|^p\log^\gamma(e+|D\bu|))\,dx\\
            & \leq \left[cA+\frac{c}{A^{\frac{1}{p-1}}} \left( \omega(R)\log^{\gamma}\left( \frac{1}{R} \right) \right)^{p'} \right]\int_{B_R}H(x,D\bu)\,dx
        \end{aligned}
    \end{equation}
    where $c$ depends on $n,N,p,\|D\bu\|_{L^p},\tilde{L}$ and $\gamma$. If $\omega(R)=0$, then $a\equiv a_i(R)$ in $B_R$, so that $\mc{D}_2=0$ and \eqref{eq: 4.8} follows from \eqref{eq: 4.12}. Otherwise we choose $A=\omega(R)\log^{\gamma}(1/R)>0$ in \eqref{eq: 4.13} and invoke \eqref{eq: 4.12}, which gives \eqref{eq: 4.8} and completes the proof.
\end{proof}

The comparison lemma yields the following decay estimate for minimizers of $\plog$.

\begin{lemma} \label{lem: 4.5}
    Let $\bu \in W^{1,p}(\Omega;\mr^N)$ be a vector-valued local minimizer of $\plog$, let $B_R \equiv B_R(x_0) \Subset \Omega$ be a ball with radius $R\leq 1/e$, and let $a_i(R)$ be as in \eqref{eq: 4.2}. Moreover, let $\omega(\cdot)$ be as in \eqref{eq: 1.6} and assume that \eqref{eq: 3.5} holds. Then the inequality
    \begin{equation} \label{eq: 4.14}
        \int_{B_\varrho}H(x,D\bu)\,dx \leq c_d\left[ \left( \frac{\varrho}{R} \right)^n+\omega(R)\log^{\gamma}\left( \frac{1}{R} \right) \right] \int_{B_R}H(x,D\bu)\,dx
    \end{equation}
    holds for every $0<\varrho\leq R$, where $B_\varrho\equiv B_\varrho(x_0)$, for a constant $c_d$ depending only on $n$, $N$, $p$, $\|D\bu\|_{L^p}$, $\tilde{L}$ and $\gamma$.   
\end{lemma}

\begin{proof}
    It suffices to prove \eqref{eq: 4.14} for $\varrho \leq R/4$, since for $\varrho>R/4$ the estimate is trivial. By Lemma~\ref{lem: 4.3} with $\tilde{\gamma}=1$ we have
    \begin{equation}    \label{eq: 4.15}
        \begin{aligned}
            \int_{B_{\varrho}} H(x,D\bu)\,dx 
            & \leq \int_{B_{\varrho}} (|D\bu|^p+a_i(R)|D\bu|^p\log^{\gamma}(e+|D\bu|))\,dx \\
            & \qquad + \int_{B_{R/2}} (a(x)-a_i(R))|D\bu|^p\log^{\gamma}(e+|D\bu|)\,dx \\
            & \leq \int_{B_{\varrho}} (|D\bu|^p+a_i(R)|D\bu|^p\log^{\gamma}(e+|D\bu|))\,dx \\
            & \qquad + c\omega(R)\log^{\gamma}\left( \frac{1}{R} \right)\int_{B_R} H(x,D\bu)\,dx
        \end{aligned}
    \end{equation}
    for some constant $c$ depending only on $n,N,p,\|D\bu\|_{L^p},\tilde{L}$ and $\gamma$. To estimate the first term on the right-hand side of \eqref{eq: 4.15}, we consider, as in Lemma~\ref{lem: 4.4}, the function $\bv$ defined in \eqref{eq: 4.7}. Applying \eqref{eq: 2.12} to $\varphi(t)=t^p$ and to $\varphi(t)=t^p\log^{\gamma}(e+t)$ separately, and recalling the definition \eqref{eq: 2.14} of $\bV_{\log^{\gamma}}$, we obtain
    \begin{equation}    \label{eq: 4.16}
        \begin{aligned}
            & \int_{B_{\varrho}} (|D\bu|^p+a_i(R)|D\bu|^p\log^{\gamma}(e+|D\bu|))\,dx\\
            & \qquad \leq c\int_{B_\varrho} (|\bV_p(D\bu)|^2+a_i(R)|\bV_{\log^{\gamma}}(D\bu)|^2)\,dx\\
            & \qquad \leq c\int_{B_\varrho} (|\bV_p(D\bv)|^2+a_i(R)|\bV_{\log^{\gamma}}(D\bv)|^2)\,dx\\
            & \qquad \qquad +c\int_{B_{R/2}}(|\bV_p(D\bu)-\bV_p(D\bv)|^2+a_i(R)|\bV_{\log^{\gamma}}(D\bu)-\bV_{\log^{\gamma}}(D\bv)|^2)\,dx\\
            & \qquad =: \mathrm{I}+\mathrm{II},
        \end{aligned}
    \end{equation}
    for a constant $c$ depending only on $n,p$ and $\gamma$. By the supremum estimate \eqref{eq: 2.23} for the minimizer $\bv$ of the frozen functional and the minimality of $\bv$, we have
    \begin{align*}
        \frac{1}{|B_\varrho|} \cdot \mathrm{I}
        & \leq c \Xint-_{B_\varrho} (|D\bv|^p+a_i(R)|D\bv|^p\log^{\gamma}(e+|D\bv|))\,dx\\
        & \leq c\sup_{B_\varrho}\left( |D\bv|^p+a_i(R)|D\bv|^p\log^{\gamma}(e+|D\bv|) \right)\\
        & \leq c\Xint-_{B_{R/2}} \left( |D\bv|^p+a_i(R)|D\bv|^p\log^{\gamma}(e+|D\bv|) \right)\,dx \\
        & \leq c\Xint-_{B_{R/2}} \left( |D\bu|^p+a_i(R)|D\bu|^p\log^{\gamma}(e+|D\bu|) \right)\,dx \\
        & \leq c\Xint-_{B_R} \left( |D\bu|^p+a_i(R)|D\bu|^p\log^{\gamma}(e+|D\bu|) \right)\,dx\\
        & \leq c\Xint-_{B_R}H(x,D\bu)\,dx
    \end{align*}
    for a constant $c$ depending only on $n$, $N$, $p$ and $\gamma$. Finally, we estimate $\mathrm{II}$ by Lemma~\ref{lem: 4.4}, and combining the resulting bound with the preceding display we arrive at \eqref{eq: 4.14}.
\end{proof}

\subsection{H\"older continuity of minima} \label{subsection 4.2}
In this subsection, we prove the first assertion of Theorem~\ref{thm: 1.2}. We consider a vector-valued local minimizer $\bu\in W^{1,p}(\Omega;\mr^N)$ of the functional $\plog$ defined in \eqref{eq: 1.2} and show that $\bu\in C_{\mathrm{loc}}^{0,\beta}(\Omega;\mr^N)$ for every $\beta\in(0,1)$, provided that $\ell=0$. More precisely, with $\omega(\cdot)$ the modulus of continuity of $a(\cdot)$ as in \eqref{eq: 1.6}, we assume that
\begin{equation} \label{eq: 4.17}
\limsup_{r\to 0}\omega(r)\log^{\gamma}\left(\frac{1}{r}\right)=0.
\end{equation}

We first recall the definition of Campanato spaces and the integral characterization of H\"older continuity, which are the main tools of the proof. For a bounded open set $\Omega'\subset\mr^n$, $x_0\in\Omega'$ and $\varrho>0$ we write $\Omega'(x_0,\varrho):=\Omega'\cap B_\varrho(x_0)$.
\begin{definition}\label{def:4.6}
    Let $\Omega'\subset\mr^n$ be a bounded open set, $p\geq1$ and $\lambda\geq0$. The Campanato space $\mc{L}^{p,\lambda}(\Omega';\mr^N)$ is the space of all $\bu \in L^p(\Omega';\mr^N)$ such that
    \begin{equation*}
        [\bu]_{\mc{L}^{p,\lambda}(\Omega';\mr^N)}^p:=\sup_{\substack{x_0 \in \Omega'\\ \varrho>0}}\varrho^{-\lambda}\int_{\Omega'(x_0,\varrho)}|\bu-(\bu)_{\Omega'(x_0,\varrho)}|^p\,dx<+\infty,
    \end{equation*}
    endowed with the norm $\|\bu\|_{\mc{L}^{p,\lambda}(\Omega';\mr^N)}:=[\bu]_{\mc{L}^{p,\lambda}(\Omega';\mr^N)}+\|\bu\|_{L^p(\Omega';\mr^N)}$.
\end{definition}

The following is Campanato's integral characterization of H\"older continuity, see \cite[Theorem~5.5]{MR3099262}.
\begin{theorem}    \label{thm: 4.7}
    Let $\Omega'\subset\mr^n$ be a ball, let $n<\lambda \leq n+p$ and $\beta = (\lambda-n)/p$. Then
    \begin{equation*}
        \mc{L}^{p,\lambda}(\Omega' ; \mr^N) \cong C^{0,\beta}(\overline{\Omega'};\mr^N),
    \end{equation*}
    and the seminorm
    \begin{equation*}
        [\bu]_{C^{0,\beta}(\Omega';\mr^N)}:= \sup_{x,y \in \Omega', x\neq y} \frac{|\bu(x)-\bu(y)|}{|x-y|^{\beta}}
    \end{equation*}
    is equivalent to $[\bu]_{\mc{L}^{p,\lambda}(\Omega';\mr^N)}$, with constants depending only on $n$, $N$, $p$ and $\lambda$, and in particular independent of the radius of $\Omega'$.
\end{theorem}
The independence of the constants of the radius of $\Omega'$ follows by scaling, since both seminorms are homogeneous of degree $-\beta$ with respect to dilations.

We now derive a Morrey type decay estimate from Lemma~\ref{lem: 4.5}.
\begin{lemma}    \label{lem: 4.8}
    Assume \eqref{eq: 4.17}. For every $\delta \in (0,n)$, there exist a constant $c$, depending only on $n$, $N$, $p$, $\|D\bu\|_{L^p}$, $\tilde{L}$, $\delta$ and $\gamma$, and a radius
    \begin{equation}    \label{eq: 4.18}
        R_0 \equiv R_0(n,N,p,\|D\bu\|_{L^p},\tilde{L},\delta,\gamma,\omega(\cdot)) \leq \frac{1}{e},
    \end{equation}
    such that the decay estimate
    \begin{equation}    \label{eq: 4.19}
        \int_{B_\varrho}H(x,D\bu)\,dx \leq c \left( \frac{\varrho}{R} \right)^{n-\delta}\int_{B_R}H(x,D\bu)\,dx
    \end{equation}
    holds for every pair of concentric balls $B_\varrho\subset B_R \Subset \Omega$ with $0<\varrho \leq R \leq R_0$.
\end{lemma}

\begin{proof}
    Let $\bar{\varepsilon}\equiv\bar{\varepsilon}(n,\delta,c_d)>0$ be the constant provided by Lemma~\ref{lem: 2.11} with $\tilde{c}=c_d$, where $c_d$ is the constant appearing in Lemma~\ref{lem: 4.5}. In view of \eqref{eq: 4.17}, we can choose $R_0\in(0,1/e]$ such that
    \begin{equation*}
        \omega(R)\log^{\gamma}\left( \frac{1}{R} \right) \leq \bar{\varepsilon} \qquad \text{for every } R \leq R_0.
    \end{equation*}
    By Lemma~\ref{lem: 4.5} and the previous choice of $R_0$, whenever $0<\varrho \leq R \leq R_0$ and $B_R\Subset\Omega$ we have
    \begin{align*}
        \int_{B_\varrho}H(x,D\bu)\,dx
        & \leq c_d\left[ \left( \frac{\varrho}{R} \right)^n+\omega(R)\log^{\gamma}\left( \frac{1}{R } \right) \right] \int_{B_R}H(x,D\bu)\,dx \\
        & \leq c_d \left[ \left( \frac{\varrho}{R} \right)^n+\bar{\varepsilon} \right] \int_{B_R}H(x,D\bu)\,dx.
    \end{align*}
    Applying Lemma~\ref{lem: 2.11} with the choice $\phi(\varrho)\equiv \int_{B_\varrho(x_0)}H(x,D\bu)\,dx$, where $x_0$ is the common center of the balls, and with $\tilde{R}=R$, we obtain \eqref{eq: 4.19}.
\end{proof}

We are now in a position to prove the first assertion of Theorem~\ref{thm: 1.2}. More precisely, we show that for every $\beta\in(0,1)$ there exists a radius $R_0\leq 1/e$, depending on $n,N,p,\|D\bu\|_{L^p},\tilde{L},\gamma,\beta$ and $\omega(\cdot)$, such that the H\"older seminorm
\begin{equation}    \label{eq: 4.20}
    [\bu]_{C^{0,\beta}(B_{R/2};\mr^N)}:=\sup_{x,y \in B_{R/2}, x\neq y} \frac{|\bu(x)-\bu(y)|}{|x-y|^{\beta}}
\end{equation}
is finite for every ball $B_R\equiv B_R(x_0)\Subset\Omega$ with $R\leq R_0$ and satisfies the quantitative estimate
\begin{equation}    \label{eq: 4.21}
    [\bu]_{C^{0,\beta}(B_{R/2};\mr^N)}\leq c\left( R^{p(1-\beta)}\Xint-_{B_R}H(x,D\bu)\,dx \right)^{\frac{1}{p}},
\end{equation}
where $c\equiv c(n,N,p,\|D\bu\|_{L^p},\tilde{L},\gamma,\beta)$.
\begin{proof}[Proof of \ref{thm: 1.2(1)}]
    Fix $\beta\in(0,1)$ and set
    \begin{equation}    \label{eq: 4.22}
        \delta:=\min\left\{ p(1-\beta),\frac{n}{2} \right\}\in(0,n), \qquad \beta':=1-\frac{\delta}{p}, \qquad \lambda:=n+p-\delta,
    \end{equation}
    so that $\beta\leq\beta'<1$ and $\lambda\in(n,n+p)$. Let $R_0\leq1/e$ be the radius given by Lemma~\ref{lem: 4.8} for this choice of $\delta$, and let $B_R\equiv B_R(x_0)\Subset\Omega$ with $R\leq R_0$. For every $y\in B_{R/2}$ and $0<\varrho\leq R/2$ we have $B_\varrho(y)\subset B_{R/2}(y)\subset B_R$ and $B_{R/2}(y)\Subset\Omega$. Therefore, Lemma~\ref{lem: 3.2} and \eqref{eq: 4.19}, applied to the concentric balls $B_\varrho(y)\subset B_{R/2}(y)$, give
    \begin{equation}    \label{eq: 4.23}
        \begin{aligned}
            \int_{B_\varrho(y)}|\bu-(\bu)_{B_\varrho(y)}|^p\,dx
            & \leq c\varrho^p \int_{B_\varrho(y)}|D\bu|^p\,dx \leq c\varrho^p \int_{B_\varrho(y)}H(x,D\bu)\,dx \\
            & \leq c\varrho^p\left( \frac{2\varrho}{R} \right)^{n-\delta}\int_{B_{R/2}(y)}H(x,D\bu)\,dx \\
            & \leq c\varrho^{\lambda}R^{-(n-\delta)}\int_{B_R}H(x,D\bu)\,dx,
        \end{aligned}
    \end{equation}
    where $c\equiv c(n,N,p,\|D\bu\|_{L^p},\tilde{L},\delta,\gamma)$. Here and in the following we use the elementary inequality $\int_E|\bu-(\bu)_E|^p\,dx\leq2^p\int_E|\bu-\mathbf{m}|^p\,dx$, which holds for every measurable set $E$ with $0<|E|<\infty$ and every $\mathbf{m}\in\mr^N$. Applied with $E=B_{R/2}\cap B_\varrho(y)$ and $\mathbf{m}=(\bu)_{B_\varrho(y)}$, it shows that the left-hand side of \eqref{eq: 4.23} controls the corresponding integral over $B_{R/2}\cap B_\varrho(y)$ for $\varrho\leq R/2$. If instead $\varrho>R/2$, then $B_{R/2}\cap B_\varrho(y)\subset B_{R/2}$, and Lemma~\ref{lem: 3.2} on $B_{R/2}$ gives
    \begin{equation}    \label{eq: 4.24}
        \begin{aligned}
            & \int_{B_{R/2}\cap B_\varrho(y)}|\bu-(\bu)_{B_{R/2}\cap B_\varrho(y)}|^p\,dx \\
            & \quad \leq 2^p\int_{B_{R/2}}|\bu-(\bu)_{B_{R/2}}|^p\,dx \leq cR^p\int_{B_{R}}H(x,D\bu)\,dx \\
            & \quad \leq c\varrho^{\lambda}R^{-(n-\delta)}\int_{B_R}H(x,D\bu)\,dx,
        \end{aligned}
    \end{equation}
    since $R<2\varrho$. Therefore $\bu\in\mc{L}^{p,\lambda}(B_{R/2};\mr^N)$ and
    \begin{equation}    \label{eq: 4.25}
        [\bu]_{\mc{L}^{p,\lambda}(B_{R/2};\mr^N)}^p \leq cR^{-(n-\delta)}\int_{B_R}H(x,D\bu)\,dx = cR^{\delta}\Xint-_{B_R}H(x,D\bu)\,dx.
    \end{equation}
    By Theorem~\ref{thm: 4.7}, applied in the ball $B_{R/2}$ with $\beta'=(\lambda-n)/p$, the function $\bu$ belongs to $C^{0,\beta'}(\overline{B_{R/2}};\mr^N)$. Since $|x-y|^{\beta'-\beta}\leq R^{\beta'-\beta}$ for all $x,y\in B_{R/2}$, we obtain
    \begin{equation*}
        \begin{aligned}
            [\bu]_{C^{0,\beta}(B_{R/2};\mr^N)}
            & \leq R^{\beta'-\beta}[\bu]_{C^{0,\beta'}(B_{R/2};\mr^N)} \\
            & \leq cR^{\beta'-\beta}[\bu]_{\mc{L}^{p,\lambda}(B_{R/2};\mr^N)} \\
            & \leq c\left( R^{p(\beta'-\beta)+\delta}\Xint-_{B_R}H(x,D\bu)\,dx \right)^{\frac{1}{p}} \\
            & = c\left( R^{p(1-\beta)}\Xint-_{B_R}H(x,D\bu)\,dx \right)^{\frac{1}{p}},
        \end{aligned}
    \end{equation*}
    for a constant $c$ depending only on $n,N,p,\|D\bu\|_{L^p},\tilde{L},\gamma$ and $\beta$. This proves \eqref{eq: 4.21}, and therefore $\bu\in C_{\mathrm{loc}}^{0,\beta}(\Omega;\mr^N)$ for every $\beta\in(0,1)$, which completes the proof of Theorem~\ref{thm: 1.2}\ref{thm: 1.2(1)}.
\end{proof}

\subsection{H\"older continuity of the gradient} \label{subsection 4.3}

In this subsection, we complete the proof of Theorem~\ref{thm: 1.2} by proving the second assertion, which concerns the H\"older continuity of the gradient of minimizers. Let $\bu \in W^{1,p}(\Omega;\mr^N)$ be a vector-valued local minimizer of the functional $\plog$ defined in \eqref{eq: 1.2}. We prove that if
\begin{equation}    \label{eq: 4.26}
    \omega(R) \leq \tilde{L}R^\sigma \quad \text{for every }R \leq 1
\end{equation}
for some $\sigma \in (0,1)$ and $\tilde{L} \geq 1$, then there exists $\beta \in (0,1)$, depending only on $n$, $N$, $p$, $\gamma$ and $\sigma$, such that
\begin{equation}    \label{eq: 4.27}
    D\bu \in C_{\mathrm{loc}}^{0,\beta}(\Omega;\mr^{N\times n}).
\end{equation}
Note that \eqref{eq: 4.26} implies \eqref{eq: 4.17}, as well as \eqref{eq: 3.5} with $\tilde{L}$ replaced by $\tilde{L}(\gamma/(e\sigma))^\gamma$, since
\begin{equation*}
    \sup_{0<R\leq1}R^{\sigma}\log^\gamma\left( \frac{1}{R} \right)=\left( \frac{\gamma}{e\sigma} \right)^{\gamma}.
\end{equation*}
In particular, all the results of the previous subsections are available, and the radius $R_0$ of Lemma~\ref{lem: 4.8} now depends only on $n$, $N$, $p$, $\|D\bu\|_{L^p}$, $\tilde{L}$, $\delta$, $\gamma$ and $\sigma$. To simplify the notation we keep denoting by $\tilde{L}$ the constant in \eqref{eq: 3.5}.

\begin{lemma}    \label{lem: 4.9}
    Let $\tilde{\alpha}\in(0,1)$ be the exponent of Theorem~\ref{thm: 2.13}. Then the inequality
    \begin{equation}    \label{eq: 4.28}
        \Xint-_{B_\varrho}|D\bu-(D\bu)_{B_\varrho}|^p \,dx \leq c\left[ \left( \frac{\varrho}{R} \right)^{\tilde{\alpha}p}+R^{\frac{\sigma}{4}}\left( \frac{R}{\varrho} \right)^n \right]\Xint-_{B_R}H(x,D\bu)\,dx
    \end{equation}
    holds for every pair of concentric balls $B_\varrho \subset B_R \Subset \Omega$ with $0<\varrho \leq R \leq 1/e$, where $c$ is a positive constant depending only on $n,N,p,\tilde{L},\|D\bu\|_{L^p},\sigma$ and $\gamma$.
\end{lemma}

\begin{proof}
    It suffices to prove \eqref{eq: 4.28} for $0<\varrho \leq R/2$, since for $R/2 <\varrho \leq R$ it holds trivially. Let $a_i(R)$ be as in \eqref{eq: 4.2} and let $\bv$ be the solution of the Dirichlet problem \eqref{eq: 4.7}. We first recall that, for every $\bz \in \mr^{N\times n}$,
    \begin{equation*}
        \Xint-_{B_\varrho}|D\bu -(D\bu)_{B_\varrho}|^p\,dx \leq c\Xint-_{B_\varrho}|D\bu-\bz|^p\,dx,
    \end{equation*}
    for a constant $c \equiv c(p)$. Using this inequality, we have
    \begin{equation}    \label{eq: 4.29}
        \begin{aligned}
            \Xint-_{B_\varrho}|D\bu -(D\bu)_{B_\varrho}|^p\,dx 
            & \leq c\Xint-_{B_\varrho}|D\bu-(D\bv)_{B_\varrho}|^p\,dx\\
            & \leq c\Xint-_{B_\varrho}|D\bv-(D\bv)_{B_\varrho}|^p\,dx+c\Xint-_{B_\varrho}|D\bu-D\bv|^p\,dx.
        \end{aligned}
    \end{equation}
    To estimate the first term on the right-hand side, we apply \eqref{eq: 2.24} to $\bv$, with $B_{R/2}$ in place of $B_R$ and with $a_0=a_i(R)$, and then use the minimality \eqref{eq: 4.10} of $\bv$. This gives, for a constant $c\equiv c(n,N,p,\gamma)$,
    \begin{equation}    \label{eq: 4.30}
        \begin{aligned}
            \Xint-_{B_\varrho}|D\bv-(D\bv)_{B_\varrho}|^p \,dx 
            & \leq c\left( \frac{\varrho}{R} \right)^{\tilde{\alpha}p}\Xint-_{B_{R/2}}(|D\bv|^p+a_i(R)|D\bv|^p\log^{\gamma}(e+|D\bv|))\,dx\\
            & \leq c\left( \frac{\varrho}{R} \right)^{\tilde{\alpha}p}\Xint-_{B_R}(|D\bu|^p+a_i(R)|D\bu|^p\log^{\gamma}(e+|D\bu|))\,dx\\
            & \leq c\left( \frac{\varrho}{R} \right)^{\tilde{\alpha}p}\Xint-_{B_R}H(x,D\bu)\,dx.
        \end{aligned}
    \end{equation}
    For the second term on the right-hand side of \eqref{eq: 4.29} we have
    \begin{equation*}
        \Xint-_{B_\varrho}|D\bu-D\bv|^p\,dx \leq c\left( \frac{R}{\varrho} \right)^n\Xint-_{B_{R/2}}|D\bu-D\bv|^p\,dx.
    \end{equation*}
    We now distinguish two cases. If $p\geq 2$, then \eqref{eq: 2.16}, Lemma~\ref{lem: 4.4} and the H\"older continuity \eqref{eq: 4.26} of $a(\cdot)$ give
    \begin{align*}
        \Xint-_{B_{R/2}} |D\bu-D\bv|^p\,dx
        & \leq c\Xint-_{B_{R/2}}|\bV_p(D\bu)-\bV_p(D\bv)|^2\,dx \\
        & \leq c \omega(R)\log^{\gamma}\left( \frac{1}{R} \right)\Xint-_{B_R}H(x,D\bu)\,dx \\
        & \leq \frac{cR^{\frac{\sigma}{2}}}{\sigma^\gamma}\Xint-_{B_R}H(x,D\bu)\,dx
    \end{align*}
    for a constant $c$ depending only on $n,N,p,\|D\bu\|_{L^p},\tilde{L}$ and $\gamma$, where we also used that
    \begin{equation*}
        \sup_{0<R\leq1}R^{\frac{\sigma}{2}}\log^{\gamma}\left( \frac{1}{R} \right)=\left( \frac{2\gamma}{e\sigma} \right)^{\gamma}.
    \end{equation*}
    If $1<p<2$, then \eqref{eq: 2.15}, H\"older's inequality with the conjugate exponents $2/p$ and $2/(2-p)$, Lemma~\ref{lem: 4.4} and again the minimality of $\bv$ give
    \begin{align*}
        & \Xint-_{B_{R/2}} |D\bu-D\bv|^p\,dx \\
        & \quad \leq c\Xint-_{B_{R/2}}|\bV_p(D\bu)-\bV_p(D\bv)|^p(|D\bu|+|D\bv|)^{\frac{p(2-p)}{2}}\,dx\\
        & \quad \leq c\left( \Xint-_{B_{R/2}} |\bV_p(D\bu)-\bV_p(D\bv)|^2\,dx \right)^{\frac{p}{2}}\left( \Xint-_{B_{R/2}} (|D\bu|+|D\bv|)^p\,dx \right)^{\frac{2-p}{2}}\\
        & \quad \leq cR^{\frac{\sigma}{4}} \Xint-_{B_R}H(x,D\bu)\,dx,
    \end{align*}
    for some $c\equiv c(n,N,p,\tilde{L},\sigma,\|D\bu\|_{L^p},\gamma)$. In both cases we conclude that
    \begin{equation}    \label{eq: 4.31}
        \Xint-_{B_\varrho}|D\bu-D\bv|^p\,dx \leq cR^{\frac{\sigma}{4}}\left( \frac{R}{\varrho} \right)^n\Xint-_{B_R}H(x,D\bu)\,dx
    \end{equation}
    for a constant $c$ depending on $n,N,p,\tilde{L},\sigma,\|D\bu\|_{L^p}$ and $\gamma$. Inserting \eqref{eq: 4.30} and \eqref{eq: 4.31} into \eqref{eq: 4.29}, we deduce \eqref{eq: 4.28}.
\end{proof}

We are now ready to prove the second assertion of Theorem~\ref{thm: 1.2}, that is, \eqref{eq: 4.27}.
\begin{proof}[Proof of \ref{thm: 1.2(2)}]
    Choosing $R=\varrho^{1-\varepsilon}$ in \eqref{eq: 4.28}, with
    \begin{equation*}
        \varepsilon := \frac{\sigma}{4\tilde{\alpha}p+\sigma+4n} \in (0,1),
    \end{equation*}
    we deduce that for every ball $B_{\varrho^{1-\varepsilon}}\equiv B_{\varrho^{1-\varepsilon}}(x_0) \Subset \Omega$ with $0<\varrho^{1-\varepsilon} \leq 1/e$ the inequalities
    \begin{equation}    \label{eq: 4.32}
        \begin{aligned}
            \Xint-_{B_\varrho}|D\bu-(D\bu)_{B_\varrho}|^p\,dx 
            & \leq c[\varrho^{\varepsilon\tilde{\alpha}p}+\varrho^{\frac{(1-\varepsilon)\sigma}{4}-\varepsilon n}] \Xint-_{B_{\varrho^{1-\varepsilon}}}H(x,D\bu)\,dx\\
            & \leq c\varrho^{\varepsilon \tilde{\alpha}p}\Xint-_{B_{\varrho^{1-\varepsilon}}}H(x,D\bu)\,dx
        \end{aligned}
    \end{equation}
    hold for a constant $c$ depending on $n,N,p,\|D\bu\|_{L^p},\tilde{L},\sigma$ and $\gamma$. Next, let $\Omega'\Subset\Omega$ be an open ball and set
    \begin{equation*}
        \delta:=\frac{\varepsilon \tilde{\alpha} p}{2(1-\varepsilon)}\in(0,n) \qquad \text{and} \qquad R_1:=\frac{1}{4}\min\{R_0,\operatorname{dist}(\Omega',\partial \Omega)\},
    \end{equation*}
    where $R_0$ is the radius determined in \eqref{eq: 4.18} for this choice of $\delta$, so that \eqref{eq: 4.19} holds for $0 < \varrho \leq R \leq R_0$. In particular, $R_1$ depends only on $n,N,p,\|D\bu\|_{L^p},\tilde{L},\sigma,\gamma$ and $\operatorname{dist}(\Omega',\partial \Omega)$. For every $x_0 \in \Omega'$ and $0 <\varrho <R_1^{1/(1-\varepsilon)}$, we have $\varrho^{1-\varepsilon} < R_1 \leq R_0$ and $B_{R_1}(x_0)\Subset\Omega$ by the definition of $R_1$. Thus, applying \eqref{eq: 4.19} with $\varrho$ and $R$ replaced by $\varrho^{1-\varepsilon}$ and $R_1$, respectively, we obtain
    \begin{equation*}
        \Xint-_{B_{\varrho^{1-\varepsilon}}}H(x,D\bu)\,dx \leq c \left( \frac{\varrho^{1-\varepsilon}}{R_1} \right)^{-\delta} \Xint-_{B_{R_1}}H(x,D\bu)\,dx.
    \end{equation*}
    Combining \eqref{eq: 4.32} with the above estimate, we conclude that
   \begin{equation} \label{eq: 4.33}
        \begin{aligned}
            \Xint-_{B_\varrho}|D\bu-(D\bu)_{B_\varrho}|^p\,dx 
            & \leq c\varrho^{\varepsilon \tilde{\alpha}p} \Xint-_{B_{\varrho^{1-\varepsilon}}}H(x,D\bu)\,dx \\
            & \leq c\varrho^{\varepsilon \tilde{\alpha}p}\left( \frac{\varrho^{1-\varepsilon}}{R_1} \right)^{-\delta} \Xint-_{B_{R_1}}H(x,D\bu)\,dx\\
            & \leq c\varrho^{\varepsilon \tilde{\alpha}p-(1-\varepsilon)\delta}\Xint-_{B_{R_1}}H(x,D\bu)\,dx\\
            & = c \varrho^{\frac{\varepsilon \tilde{\alpha} p}{2}}\Xint-_{B_{R_1}}H(x,D\bu)\,dx.
        \end{aligned}
    \end{equation}
    Multiplying both sides of the above inequality by $|B_\varrho|\varrho^{-\varepsilon \tilde{\alpha}p/2}$, we obtain
    \begin{equation} \label{eq: 4.34}
        \varrho^{-\left( n+\frac{\varepsilon\tilde{\alpha}p}{2} \right)}\int_{B_\varrho}|D\bu-(D\bu)_{B_{\varrho}}|^p\,dx \leq c\left( \frac{1}{R_1} \right)^n \int_{B_{R_1}}H(x,D\bu)\,dx.
    \end{equation}
    Note that \eqref{eq: 4.34} holds for every $x_0\in\Omega'$ and every $0<\varrho<\varrho_1:=R_1^{1/(1-\varepsilon)}$. Set $\lambda:=n+\varepsilon\tilde{\alpha}p/2\in(n,n+p)$. For $\varrho\geq\varrho_1$ we estimate
    \begin{equation*}
        \begin{aligned}
            \varrho^{-\lambda}\int_{\Omega'(x_0,\varrho)}|D\bu-(D\bu)_{\Omega'(x_0,\varrho)}|^p\,dx
            & \leq 2^p\varrho_1^{-\lambda}\int_{\Omega'}|D\bu|^p\,dx \\
            & \leq 2^p\varrho_1^{-\lambda}\int_{\Omega}H(x,D\bu)\,dx,
        \end{aligned}
    \end{equation*}
    while for $\varrho<\varrho_1$ we use the inequality $\int_E|D\bu-(D\bu)_E|^p\,dx\leq2^p\int_E|D\bu-\mathbf{m}|^p\,dx$ with $E=\Omega'(x_0,\varrho)$ and $\mathbf{m}=(D\bu)_{B_\varrho(x_0)}$ together with \eqref{eq: 4.34}. Therefore $D\bu\in\mc{L}^{p,\lambda}(\Omega';\mr^{N\times n})$ and
    \begin{equation} \label{eq: 4.35}
        [D\bu]_{\mc{L}^{p,\lambda}(\Omega';\mr^{N\times n})}^p \leq c\left( R_1^{-n}+\varrho_1^{-\lambda} \right)\int_{\Omega}H(x,D\bu)\,dx.
    \end{equation}
    Theorem~\ref{thm: 4.7}, applied with $\beta=(\lambda-n)/p=\varepsilon\tilde{\alpha}/2$, now yields $D\bu\in C^{0,\beta}(\overline{\Omega'};\mr^{N\times n})$ and
    \begin{equation*}
        [D\bu]_{C^{0,\beta}(\Omega';\mr^{N\times n})} \leq c[D\bu]_{\mc{L}^{p,\lambda}(\Omega';\mr^{N\times n})} \leq c\left(\int_{\Omega}H(x,D\bu)\,dx \right)^{\frac{1}{p}},
    \end{equation*}
    where $c\equiv c(n,N,p,\|D\bu\|_{L^p},\tilde{L},\sigma,\gamma,\operatorname{dist}(\Omega',\partial\Omega))$. Since every compact subset of $\Omega$ is covered by finitely many such balls $\Omega'$, we conclude that $D\bu \in C_{\mathrm{loc}}^{0,\beta}(\Omega;\mr^{N\times n})$ with $\beta=\varepsilon\tilde{\alpha}/2$, which depends only on $n,N,p,\gamma$ and $\sigma$. This proves \eqref{eq: 4.27} and completes the proof of Theorem~\ref{thm: 1.2}.
\end{proof}

\medskip
\noindent\textbf{Acknowledgments.} This work is supported by the National Research Foundation of Korea (NRF) grant funded by the Korea government [Grant Nos. RS-2025-00555316 and RS-2025-25415411].

\providecommand{\bysame}{\leavevmode\hbox to3em{\hrulefill}\thinspace}
\providecommand{\MR}{\relax\ifhmode\unskip\space\fi MR }
\providecommand{\MRhref}[2]{%
  \href{http://www.ams.org/mathscinet-getitem?mr=#1}{#2}
}
\providecommand{\href}[2]{#2}

\end{document}